\documentclass[a4paper, 11pt, reqno]{amsart}
\usepackage[margin=1.2in]{geometry}

\usepackage{amssymb}
\usepackage{amsmath}
\usepackage{mathrsfs}
\usepackage[all]{xy}
\usepackage{mathtools}
\usepackage[dvipsnames]{xcolor}
\usepackage[english]{babel}
\usepackage{enumitem}

\usepackage{colonequals}
\usepackage{tikz}
\usepackage{tikz-cd}
\usepackage{pgf,tikz}
\usetikzlibrary{arrows}

\usepackage[style=alphabetic, sorting=nyt, maxbibnames=6]{biblatex}
\usepackage{csquotes}              

\usepackage[hidelinks]{hyperref}
\usepackage{url}
\usepackage[hyphenbreaks]{breakurl}

\usepackage[OT2,T1]{fontenc}
\DeclareSymbolFont{cyrletters}{OT2}{wncyr}{m}{n}

\usepackage[skip=10pt plus2pt]{parskip} 

\numberwithin{equation}{section} \numberwithin{figure}{section}

\DeclareSymbolFont{cyrletters}{OT2}{wncyr}{m}{n}
\DeclareMathSymbol{\Sha}{\mathalpha}{cyrletters}{"58}
\DeclareMathSymbol{\Be}{\mathalpha}{cyrletters}{"42}

 \newcommand{\term}[1]{\textit{#1}}          

\DeclareMathOperator{\ord}{ord}
\DeclareMathOperator{\Pic}{Pic}
\DeclareMathOperator{\Div}{Div}
\DeclareMathOperator{\divv}{div}
\DeclareMathOperator{\Gal}{Gal}

\DeclareMathOperator{\supp}{supp}

\newcommand{\del}{\partial}

\def \C {{\mathbb C}}

\def \F {{\mathbb F}}

\def \P {{\mathbb P}}
\def \Q {{\mathbb Q}}
\def \R {{\mathbb R}}

\def \Z {{\mathbb Z}}

\def \cF {{\mathcal F}}
\def \cG {{\mathcal G}}

\def \cO {{\mathcal O}}

\newcommand{\bt}{{\boldsymbol t}}

\newcommand{\br}{{\boldsymbol r}}
\newcommand{\bF}{{\boldsymbol F}}
\newcommand{\bG}{{\boldsymbol G}}
\newcommand{\bX}{{\boldsymbol X}}
\newcommand{\bY}{{\boldsymbol Y}}
\newcommand{\bZ}{{\boldsymbol Z}}
\newcommand{\bT}{{\boldsymbol T}}

\newcommand{\ra}{\rightarrow}

\newcommand\blfootnote[1]{%
  \begingroup
  \renewcommand\thefootnote{}\footnote{#1}%
  \addtocounter{footnote}{-1}%
  \endgroup
}

\newtheorem{lemma}{Lemma}

\newtheorem{theorem}[lemma]{Theorem}
\newtheorem{proposition}[lemma]{Proposition}
\newtheorem{corollary}[lemma]{Corollary}

\theoremstyle{definition}
\newtheorem{example}[lemma]{Example}

\newtheorem{definition}[lemma]{Definition}
\newtheorem{remark}[lemma]{Remark}

\numberwithin{lemma}{section}

\begin{document}

\title{A naive $p$-adic height on the Jacobians of curves of genus 2}

\author{Manoy T. Trip}

\begin{abstract}
Consider a genus 2 curve defined over $\Q$ given by an affine equation of the form $y^2 = f(x)$ for some polynomial $f$ of degree 5, and let $p$ be an odd prime. Extending work of Perrin-Riou for elliptic curves, we construct a naive $p$-adic height function on a finite index subgroup of the Jacobian $J$ of this curve, using the explicit embedding of $J$ in $\P^8$ and the associated formal group described by Grant. We use the naive height to construct a global height $h_p \colon J(\Q) \ra \Q_p$ using a limit construction analogous to Tate's construction of the N\'{e}ron--Tate height, and show that it is quadratic. We then compare $h_p$ to a $p$-adic height constructed in a different way by Bianchi and show that they are equal.
\end{abstract}

\maketitle

\blfootnote{\textit{Date}: November 4, 2025.}

\thispagestyle{empty}

\section{Introduction} \label{sec:intro}

The most straightforward occurrence of a real height function on abelian varieties is what we call the naive height function on elliptic curves, based on the standard height on projective space:
\begin{align*}
    h^{\operatorname{naive}}_E \colon E(\Q) &\ra \R\\
    [x:y:z] &\mapsto \log \max\{|x|_\infty, |z|_\infty\},
\end{align*}
where $[x : y : z]$ are the projective coordinates of a point on the elliptic curve $E$ with respect to a Weierstrass model of $E$ in $\P^2$, and $|\cdot|_\infty$ is the standard archimedean absolute value on $\Q$. This height is for example used to prove the Mordell--Weil theorem for elliptic curves (e.g. \cite[VIII, Theorem 6.7]{silverman1}). This height function is close to being a quadratic form, and it can be used to construct an actual quadratic form. We refer to $h_E^{\operatorname{naive}}$ as the naive height because of its simple description and to distinguish it from this resulting quadratic height. The quadratic height is called the N\'{e}ron--Tate height and is defined as follows:
\begin{equation}\label{eq:tate}
    \begin{aligned}
        h_E \colon E(\Q) &\ra \R\\
        h_E(P) &\colonequals \lim_{n \ra \infty} \frac{1}{2^{2n}}h^{\operatorname{naive}}_E(2^n P).
    \end{aligned}
\end{equation}
This construction was given by Tate (unpublished). Around the same time, the same height was constructed by N\'{e}ron \cite{neron} as a sum of local heights. At each place $v$ of $\Q$ he defines a local height 
\[\lambda_{E, v} \colon E(\Q_v) \setminus \{\cO\} \ra \R\]
which is not quadratic but is ``quasi-quadratic'' (see \cite[Chapter VI]{silverman2}), where $\cO$ is the identity element of the group law on $E$. Then
\begin{align*}
    h_E(P) = \sum_{v \in M_\Q} \lambda_{E, v}(P)
\end{align*}
for $P \in E(\Q) \setminus \{\cO\}$, where the sum runs over all places of $\Q$. The description of the local height at the archimedean place uses the Weierstrass sigma function. 

These constructions can be adapted to define quadratic heights on the Jacobian $J$ of a curve $C$ of genus 2. We restrict to curves over $\Q$ that can be given by an affine equation of the form
\begin{equation}\label{eq:genus2model}
    y^2 = x^5 + f_4 x^4 + f_3 x^3 + f_2 x^2 + f_1 x + f_0.
\end{equation}
Such a model exists whenever $C$ has a rational Weierstrass point, which in the model becomes the point at infinity $\infty$. We assume that the coefficients $f_i$ are in $\Z$. Let $\Theta$ denote the theta divisor on $J$, which we take to be the image of the curve $C$ under the Abel--Jacobi map with respect to the point $\infty$. For a divisor $D \in \Div(J)$, we write $\supp(D) \subseteq J$ for the closed subset of points that lie on one of the components of $D$. Then, similar to \eqref{eq:tate}, a quadratic height function
\begin{align*}
    h_J \colon J(\Q) \ra \R
\end{align*}
can be defined by a limit construction using a naive height function based on the standard height on projective space, evaluated on the image of a point in $J(\Q)$ on the associated Kummer surface as embedded in $\P^3$ \cite[335]{flynnsmart} (see \S\ref{sec:genus2curves}). It can also be recovered as the sum of local heights:
\begin{align*}
    h_J(P) = \sum_{v \in M_\Q} \lambda_{J, v}(P)
\end{align*}
for $P \in J(\Q) \setminus \supp(\Theta)$, with $\lambda_{J, v} \colon J(\Q_v)\setminus \supp(\Theta) \ra \R$ (see for example \cite[\S 5]{flynnsmart}, or \cite{lang} for the case of general abelian varieties). 

Now let us fix a prime $p$. A $p$-adic height function is a $\Q_p$-valued analogue of the real height functions we have seen so far. Similar to how real height functions arise in the Birch and Swinnerton-Dyer conjecture (see \cite[Conjecture 1.1]{balakrishnan}), there is a $p$-adic analogue of this conjecture in which $p$-adic height functions appear (see \cite{teitelbaum}, \cite[Conjecture 1.3, 1.4]{balakrishnan}). Furthermore, $p$-adic heights are used in the quadratic Chabauty method (see \cite{balakrishnanChabauty}). Usually, $p$-adic height functions are defined using a local decomposition. Indeed, one way to construct a $p$-adic height is to write it as a sum of local $p$-adic height functions, where the archimedean local $p$-adic height is trivial, and the local $p$-adic height at the prime $p$ is defined in terms of a $p$-adic sigma function. For elliptic curves, such $p$-adic sigma functions $\sigma_p^{(s)}$, which are defined up to a scalar choice $s \in \Q_p$, have been described for instance by Bernardi \cite{bernardi}, N\'{e}ron \cite{neron2} and Mazur--Tate \cite{mazurtate} (see this last reference for a broader treatment of the existing literature on $p$-adic sigma functions). Bernardi defined the function for $s = 0$, based on the Taylor expansion of the complex Weierstrass sigma function. In general, these $p$-adic sigma functions can be defined as formal power series which converge on a finite index subgroup of $E(\Q_p)$ contained in the formal group associated to $E/\Q_p$. We obtain quadratic $p$-adic height functions $h_p^{(s)} \colon J(\Q) \ra \Q_p$ as a sum of local height functions, which depend on the choice of $s$. 

A similar construction was found on Jacobians of genus 2 curves by Bianchi \cite{bianchi3}. She defines local $p$-adic height functions where again the local height at $p$ depends on a choice of $p$-adic sigma function. In the genus 2 setting, these sigma functions can be defined, generalising work of Blakestad \cite{blakestad}, up to a choice of a symmetric matrix $\boldsymbol{c} = (c_{ij})_{i, j} \in M_{2 \times 2}(\Q_p)$, and are denoted $\sigma_p^{(\boldsymbol{c})}(t_1, t_2) \in \Q_p[[t_1, t_2]]$\cite[\S 3.2]{bianchi3}. The functions are made explicit using the description of the formal group corresponding to the Jacobian due to Grant \cite{grant} (see \S\ref{sec:formalgroupjacobian}). For the choice $\boldsymbol{c} = \boldsymbol{0}$, the zero matrix, we call $\sigma_p^{(\boldsymbol{0})}$ the naive sigma function, because as a power series it satisfies differential equations analogous to those for the complex sigma function, see Remark \ref{rem:padicsigmafunctions}.

We saw that in the real setting, we have two different constructions to obtain the same quadratic height function. The natural question arises as to whether it is possible to similarly construct $p$-adic height functions from a naive height which has a simpler description. For elliptic curves, this was done by Perrin-Riou \cite{perrinriou}. She defined a naive $p$-adic height on elliptic curves which can be used to construct a global quadratic $p$-adic height using a limit, analogous to Tate's construction of the real quadratic height in \eqref{eq:tate}. She considers a Weierstrass model for an elliptic curve $E/\Q$ given by an affine equation \[y^2 + a_1 xy + a_3 y = x^3 + a_2 x^2 + a_4 x + a_6\] with $a_i \in \Z$, and the natural map $\kappa_E \colon E \ra \P^1$ to the corresponding Kummer variety, which is defined by $(x, y) \mapsto [x : 1]$ and $\cO \mapsto [1 : 0]$. For any $P \in E(\Q)$, we can write $\kappa_E(P) = [w_1(P) : w_2(P)]$ normalised such that $w_1(P)$ and $w_2(P)$ are coprime integers. Then Perrin-Riou defines a naive height:
\begin{equation}\label{eq:naivepadicheightE}
    \begin{aligned}
        H_2 \colon E_p(\Q) \setminus \{\cO\} &\ra \Q_p\\
         H_2(P) &= \log_p(w_1(P)),
    \end{aligned}
\end{equation}
where $E_p(\Q)$ is a subgroup of the kernel of reduction of the Weierstrass model of $E$ modulo the prime $p$, which has finite index in $E(\Q)$. The function $\log_p$ is a branch of the $p$-adic logarithm (see \S \ref{sec:preliminaries}). By defining \[h_2(P) \colonequals \lim_{n \ra \infty} \frac{1}{p^{2n}} H_2(p^n P)\] she obtains a quadratic height, and shows a relation between $h_2$ and the height obtained from local heights using the naive sigma function $\sigma_p^{(0)}$ defined by Bernardi \cite[Th\'{e}or\`{e}me]{perrinriou2}. In particular, if $a_1 = a_2 = 0$, then these two heights are the same. 

The analogous problem in the setting of Jacobians of genus 2 curves is treated in this paper. We adapt the methods of Perrin-Riou to define a naive $p$-adic height on the Jacobians of genus 2 curves with a rational Weierstrass point, and show that the global $p$-adic height we obtain agrees with the global $p$-adic height of Bianchi using the naive sigma function $\sigma_p^{(\boldsymbol{0})}$. It follows that we produced an alternative derivation for Bianchi's height, based on more elementary methods, which can therefore be of interest to a more general audience. We then more generally give a description of a naive height for any choice of $\boldsymbol{c}$ that recovers Bianchi's height based on the corresponding choice of $p$-adic sigma function $\sigma_p^{(\boldsymbol{c})}$.

Let us consider a genus 2 curve $C$ as in \eqref{eq:genus2model} with Jacobian $J$ and an odd prime number $p$. We consider a finite index subgroup $J_p \subseteq J(\Q)$ which is contained in the kernel of reduction of $J$ at the prime $p$ (see \eqref{eq:defJp}). We define a naive $p$-adic height function $H_p \colon J_p \ra \Q_p$ (see Definition \ref{def:Hp}) and use this to define a quadratic height $h_p$ on $J(\Q)$. Our main results are summarised in the following two theorems.

\begin{theorem}\label{thm:mainthm}
Let $P \in J_p$. Then the following limit exists:
\begin{equation}\label{eq:quadraticheighthpJ}
    h_p(P) \colonequals \lim_{n \ra \infty} \frac{1}{p^{2n}} H_p(p^n P).
\end{equation}
This uniquely extends to a function $h_p \colon J(\Q) \ra \Q_p$ which is a quadratic form.
\end{theorem}

\begin{theorem}\label{thm:comparison}
    Let $h_p \colon J(\Q) \ra \Q_p$ be the quadratic height function in Theorem \ref{thm:mainthm}, and let $h_p^ {(\boldsymbol{0})} \colon J(\Q) \ra \Q_p$ be the quadratic height defined by Bianchi in \cite[Definition 4.8]{bianchi3} using the naive $p$-adic sigma function $\sigma_p^{(\boldsymbol{0})}$. Then $h_p = h_p^ {(\boldsymbol{0})}$ on all of $J(\Q)$.
\end{theorem}

In \S\ref{sec:preliminaries}, we give a description of the Jacobian of a genus $2$ curve of the form \eqref{eq:genus2model}, and of the corresponding Kummer surface. We recall some properties of formal groups and see how the kernel of reduction of the Jacobian can be identified with a formal group. In \S\ref{sec:naiveheight} we show the existence of the limit in Theorem \ref{thm:mainthm}, and we show the quadraticity of the resulting global height function in \S\ref{sec:quadraticity}. Finally, we show the equality of $h_p$ and $h_p^{(\boldsymbol{0})}$ in \S\ref{sec:comparison}, and describe how we can generalise the naive height definition to also recover the heights $h_p^{(\boldsymbol{c})}$ for other choices of $\boldsymbol{c}$.

\subsection*{Acknowledgements}
I want to thank Steffen M\"{u}ller and Francesca Bianchi for proposing this topic to me and more generally for introducing me to the research area of heights and $p$-adic analysis. We have had many useful discussions and both were very generous with their time and feedback, for which I am very grateful. I also thank Oliver Lorscheid for helpful comments on an earlier version of the manuscript. I want to thank the anonymous referee for their detailed feedback and useful suggestions, in particular for suggesting the generalisation in Theorem \ref{thm:generalnaiveheight} and the geometric argument in Remark \ref{rem:alternativeproofLemma}.

\section{Preliminaries}\label{sec:preliminaries}
Let $q$ be a prime. We will write $\ord_q$ for the $q$-adic valuation on $\Q_q$, and $|\cdot|_q$ for the corresponding $q$-adic absolute value, normalised such that $|q|_q = q^{-1}$. Together with the standard archimedean absolute value $|\cdot|_\infty$, these absolute values satisfy the product formula, $\prod_{v \in M_\Q} |a|_v = 1$ for all $a \in \Q^\times$. We will also make use of the $q$-adic logarithm function for odd primes $q$, which we denote by $\log_q \colon \Q_q^\times \rightarrow \Q_q$. For $a \in 1 + q\Z_q$, this logarithm is defined by the usual power series $$\log_q(a) = \sum_{n=1}^\infty \frac{(-1)^{n+1}}{n}(a-1)^n.$$ We extend it uniquely to $\Z_q^\times$ by requiring $\log_q(a) + \log_q(b) = \log_q(ab)$ for all $a, b \in \Z_q^\times$, and using the fact that any $a \in \Z_q^\times$ can be written as $a = cr$ for some unique $c \in 1 + q\Z_q$ and $r \in \Z_q^\times$ a $(q-1)$-st root of unity. We further extend $\log_q$ to $\Q_q^\times$ by setting $\log_q(q) = 0$ (the Iwasawa logarithm). It satisfies
\begin{align}
    \log_q(1 + q^k\Z_q) \subseteq q^k\Z_q\label{eq:logarithm}
\end{align}
for $k  \in \Z_{>0}$.

\subsection{Jacobians of genus 2 curves and their Kummer surfaces}\label{sec:genus2curves}
Let us fix an odd prime $p$. Let $C$ be a smooth projective curve of genus $2$ defined over $\Q$, given as the normalisation of a projective curve with affine equation of the form 
\[y^2 = x^5 + f_4 x^4 + f_3 x^3 + f_2 x^2 + f_1 x + f_0.\]
We assume that the coefficients $f_i$ are in $\Z$. The curve $C$ has one point at infinity $\infty$. Let us denote the Jacobian of $C$ by $J$ and its set of rational points by $J(\Q)$. The Jacobian $J$ can be identified with $\Pic_C^0$, and we can embed $C$ into $J$ via the Abel--Jacobi map $\Phi_\infty \colon C \rightarrow \Pic_C^0$ given by $P \mapsto [P - \infty]$, where $[D]$ denotes the linear equivalence class of a divisor $D$ on $C$. The image of this map is a subvariety $\Theta$ of $J$ of codimension one, which is called the theta divisor. Because of the identification of $J$ with $\Pic_{C}^0$, each element of $J(\Q)$ can be represented by a pair of points $\{P_1, P_2\}$ of $C$ which is fixed under the action of $\Gal(\overline{\Q}/\Q)$. We consider the Riemann-Roch space $\mathcal{L}(3\Theta)$, which is the space of rational functions $g$ on $J$ such that $3 \Theta + \divv(g)$ is an effective divisor. This is a $9$-dimensional $\overline{\Q}$-vector space with a basis \[\{1, \wp_{11}, \wp_{12}, \wp_{22}, \wp_{111}, \wp_{112}, \wp_{122}, \wp_{222}, \wp\}\] of functions that are defined over $\Q$, described explicitly in \cite[99]{grant}. We can define an embedding of $J$ into $\P^8$ using these functions as follows:
\begin{align*}
    i \colon \Pic_C^0 &\ra \P^{8}\\
    [P_1 + P_2 - 2 \infty] &\mapsto [1: \wp_{11}: \wp_{12}: \wp_{22}: \tfrac{1}{2}\wp_{111}: \tfrac{1}{2}\wp_{112}: \tfrac{1}{2}\wp_{122}: \tfrac{1}{2}\wp_{222}: \tfrac{1}{2}(\wp + f_3 \wp_{12} - f_1)].
\end{align*}
Explicit defining equations for the image of this map as a projective variety are given in \cite[Corollary 2.15]{grant}. For the nine projective coordinates of $\P^8$, we use the notation $X_0$, $X_{11}$, $X_{12}$, $X_{22}$, $X_{111}$, $X_{112}$, $X_{122}$, $X_{222}$, $X$ in this order, after \cite{grant}. As rational functions, we then have
\begin{align*}
    \wp_{ij} = \frac{X_{ij}}{X_0}, \ \wp_{ijk} = 2\frac{X_{ijk}}{X_0}\ \text{and} \ \wp = \frac{2X - f_3 X_{12}+ f_1 X_0}{X_0}.
\end{align*}
We note that $i([0]) = [0 : 0 : 0 : 0 : 1 : 0 : 0 : 0 : 0] \equalscolon \cO$.

As in the case of Perrin-Riou's naive height on elliptic curves, we want to define a naive height on $J$ such that the height of a point $P$ only depends on the image of $P$ on the Kummer variety of $J$. Let us denote the Kummer surface of $J$ by $K$. It can be embedded in $\P^3$ using a basis of $\mathcal{L}(2\Theta)$ (see \cite[Theorem 4.8.1]{birkenhake}). The even functions $1$, $\wp_{11}$, $\wp_{12}$ and $\wp_{22}$ form such a basis (\cite[\S 2]{flynn2}), and define a map to $\P^3$ as follows:
\begin{equation}\label{eq:kummermap}
    \begin{aligned}
        \Pic_C^0 &\ra \P^{3}\\
        [P_1 + P_2 - 2\infty] &\mapsto [1: \wp_{22}: -\wp_{12}:\wp_{11}].
    \end{aligned}    
\end{equation}

The image of this map gives an embedding of $K$ in $\P^3$ and can be defined by a homogeneous equation of the form 
\begin{equation*}
    G(X, Y, Z, W) \colonequals R(X, Y, Z) W^2 + S(X, Y, Z) W + T(X, Y, Z) = 0,
\end{equation*}
where $R$, $S$ and $T$ are homogeneous polynomials with coefficients in $\Z$ of total degree 2, 3, and 4, respectively. The explicit equation can be found in \cite[Appendix A]{flynn2}. We write $\kappa \colon J \ra \P^3$ for the map from the Jacobian to this explicit embedding of the Kummer surface $K$ in $\P^3$.

Let $P, Q \in J$. On the Kummer surface, we cannot in general determine $\kappa(P+Q)$ from $\kappa(P)$ and $\kappa(Q)$, but it is possible to define a multiplication-by-$m$ map $\mu_m$ on the Kummer surface for any $m \in \Z$, such that $\mu_m(\kappa(P)) = \kappa(mP)$. The limit formula for the quadratic height \eqref{eq:quadraticheighthpJ} involves the evaluation of $H_p$ on multiples of a point $P$, and to simplify this limit formula we use the equations of $\mu_m$ to express $\kappa(mP)$ in terms of the coordinates of $\kappa(P)$. In \cite[Appendix C]{flynn2}, explicit homogeneous polynomials $\delta_1, \ldots, \delta_4 \in \Z[k_1, k_2, k_3, k_4]$ of total degree 4 are given, such that for all $P \in J$ with $\kappa(P) = [w_1 : w_2 : w_3 : w_4]$, we have \[\kappa(2P) = \left[\delta_1(w_1, w_2, w_3, w_4): \delta_2(w_1, w_2, w_3, w_4): \delta_3(w_1, w_2, w_3, w_4): \delta_4(w_1, w_2, w_3, w_4)\right].\]

For general multiples, Uchida gives an explicit description which we reproduce in Theorem \ref{thm:definitionmu}. It uses the existence of biquadratic forms $B_{ij}$ satisfying the properties in the following theorem.

\begin{theorem}[{\cite[Theorem 3.4.1]{casselsflynn}}]\label{thm:biquadraticforms}
    For $i, j \in \{1, 2, 3, 4\}$, there exist polynomials $B_{ij}$ with coefficients in $\Z$, which are biquadratic in two sets of variables $k_1, \ldots, k_4$ and $l_1, \ldots, l_4$, with the following property. For any $P, Q \in J$, let us fix Kummer coordinates 
    \begin{align*}
        &\kappa(P) = [x_1: \cdots : x_4], \ \kappa(Q) = [y_1: \cdots : y_4],\\
        &\kappa(P + Q) = [z_1: \cdots : z_4], \ \kappa(P - Q) = [w_1: \cdots : w_4].
    \end{align*}
     Then there exists a constant $c \in \overline{\Q}^\times$ such that for all $i, j \in \{1, 2, 3, 4\}$, we have
    \begin{align*}
        z_i w_j + w_i z_j = 2c B_{ij}((x_1, x_2, x_3, x_4), (y_1, y_2, y_3, y_4)).
    \end{align*}
\end{theorem}
Explicit formulas for the forms $B_{ij}$ can be found in \cite{biquadratic}.

\begin{theorem}[{\cite[Theorem 3.3, Proposition 3.6, Lemma 3.8]{uchida}}]\label{thm:definitionmu}
    For any $m \geq 0$ and $i = 1, 2, 3, 4$, there exist homogeneous polynomials $\mu_{m,i} \in \Z[k_1, k_2, k_3, k_4]$ of total degree $m^2$ such that:
    \begin{align*}
        \mu_{0,1} &= \mu_{0, 2} = \mu_{0, 3} = 0,\\
        \mu_{0, 4} &= 1,\\
        \mu_{1, i} &= k_i,\\
        \mu_{2m, i} &= \delta_i(\mu_m)  \qquad \qquad \; \text{for $m \geq 1$,}\\
        \mu_{2m+1, i}k_i &= B_{ii}(\mu_{m+1}, \mu_m)  \quad \text{for $m \geq 1$},
    \end{align*}
    where we write $\mu_m = (\mu_{m,1}, \ldots, \mu_{m, 4})$. For all $P \in J(\overline{\Q})$ with $\kappa(P) = [w_1: w_2: w_3: w_4]$, we have $$\kappa(mP) = [\mu_{m,1}(w_1, \ldots, w_4): \cdots :\mu_{m,4}(w_1, \ldots, w_4)].$$
\end{theorem}

\begin{remark}
    We note that everything discussed in this section can also be done if $C$ is a genus 2 curve over any field $k$ of characteristic different from $2$. 
\end{remark}

Consider $P \in J(\Q)$ and write $\kappa(P) = [w_1: w_2: w_3: w_4]$ normalised such that $w_i \in \Z$ and $\gcd(w_1, w_2, w_3, w_4) = 1$. Then we use the notation $w_i(P) \colonequals w_i$ and 
\begin{equation}
   w(P) = (w_1(P), w_2(P), w_3(P), w_4(P)).\label{eq:normalizationkappa}
\end{equation}
Note that the normalisation is only uniquely defined up to sign, so also the values $w_i(P)$ are unique up to sign. However, we will only ever use these coordinates in the context of ratios, $v$-adic absolute values, or $p$-adic logarithms, in which case the sign is irrelevant. 

\subsection{Formal groups}\label{sec:formalgroups}

Let $R$ denote a commutative ring with identity. We recall some definitions and properties of formal groups in two dimensions, see \cite[\S 9]{hazewinkel} for a general introduction.
\begin{definition}
     A \term{$2$-parameter formal group} $\mathcal{F}$ over $R$ is a pair $\bF = (F_1, F_2)$ of power series in $R[[X_1, X_2, Y_1, Y_2]]$ with the following properties:
    \begin{enumerate}
        \item $F_i(X_1, X_2, Y_1, Y_2) = X_i + Y_i + (\text{terms of total degree } \geq 2)$.
        \item $\bF(\bX, \bF(\bY, \bZ)) = \bF(\bF(\bX, \bY), \bZ)$.
    \end{enumerate}
    We call $\bF(\bX, \bY)$ the \term{formal group law} of $\cF$. We denote the formal group by $(\cF, \bF)$ if we want to make the formal group law explicit. We say $\cF$ is \term{commutative} if $\bF$ furthermore satisfies $\bF(\bX, \bY) = \bF(\bY, \bX)$.
\end{definition}

Formal group laws satisfy group-like properties: we have $\bF(\bX, \boldsymbol{0}) = \bX$ and $\bF(\boldsymbol{0}, \bY) = \bY$, and there is a unique pair of power series $\boldsymbol{i}(\bT) \in R[[T_1, T_2]]^2$, called the \term{formal inverse}, satisfying $\bF(\bT, \boldsymbol{i}(\bT)) = \bF(\boldsymbol{i}(\bT), \bT) = \boldsymbol{0}$.

\begin{definition}
    Let $(\cF, \bF)$ and $(\cG, \bG)$ be $2$-parameter formal groups over $R$. A \term{formal group homomorphism} $\boldsymbol{f} = (f_1, f_2)$ from $\cF$ to $\cG$ defined over $R$ is a pair of power series $f_1, f_2 \in R[[T_1, T_2]]$ with no constant term, satisfying $$\boldsymbol{f}(\bF(\bX, \bY)) = \bG(\boldsymbol{f}(\bX), \boldsymbol{f}(\bY)).$$
\end{definition}

\begin{example}\label{ex:multiplicationbym}
    Let $(\cF, \bF)$ be a commutative $2$-parameter formal group. We can define the \term{multiplication-by-$m$ maps} for $m \in \Z_{\geq 0}$ inductively as homomorphisms $[m] \colon \cF \ra \cF$ by
    \begin{align*}
        [0](\bT) &= \boldsymbol{0},\\
        [m+1](\bT) &= \bF([m](\bT), \bT).
    \end{align*}
    It follows from an induction argument that they are of the form $$[m](\bT) = m\bT + (\text{terms of total degree $\geq 2$}).$$
\end{example}

Let us now assume that $R$ is a complete local commutative ring with maximal ideal $\mathfrak{m}$, and let $(\cF, \bF)$ be a $2$-parameter formal group over $R$. In this case, if we take $\boldsymbol{r} = (r_1, r_2)$ and $\boldsymbol{s} = (s_1, s_2)$ in $\mathfrak{m} \times \mathfrak{m}$, then $\bF(\boldsymbol{r}, \boldsymbol{s})$ converges in $\mathfrak{m} \times \mathfrak{m}$ by the completeness of $R$. Similarly, $\boldsymbol{i}(\boldsymbol{r})$ converges in $\mathfrak{m} \times \mathfrak{m}$. Hence we can define a \term{group associated to $\cF$}, denoted by $\cF(\mathfrak{m} \times \mathfrak{m})$, with underlying set $\mathfrak{m} \times \mathfrak{m}$, group operations $\boldsymbol{r} + \boldsymbol{s} = \bF(\boldsymbol{r}, \boldsymbol{s})$ and $-\boldsymbol{r} = \boldsymbol{i}(\boldsymbol{r})$ for all $\boldsymbol{r}, \boldsymbol{s} \in \mathfrak{m} \times \mathfrak{m}$, and identity element $(0,0)$. A formal group homomorphism also naturally defines a group homomorphism of the associated groups. For instance, the group homomorphism induced by the homomorphism $[m]$ is the usual {multiplication-by-$m$} homomorphism (see \cite[Proposition 1.6.7, Proposition 1.6.8]{tripThesis}). The following result follows straightforwardly from our description of the homomorphism $[m]$ in Example \ref{ex:multiplicationbym} (see \cite[Proposition 1.6.9]{tripThesis} for an explicit argument).

\begin{lemma}\label{lem:multipleofpninformalgroup}
    Let $p$ be a prime. Let $(\cF, \bF)$ be a commutative $2$-parameter formal group over $\Z_p$. For all $\br \in p\Z_p \times p\Z_p$ and $m \in \Z_{\geq 0}$, we have \[[p^m](\br) \in p^{m+1}\Z_p \times p^{m+1}\Z_p.\]
\end{lemma}

\subsection[The formal group associated to the Jacobian of a genus \texorpdfstring{$2$}{2} curve]{The formal group associated to the Jacobian of a genus $\boldsymbol{2}$ curve}\label{sec:formalgroupjacobian}

Consider a smooth projective curve $C$ of genus $2$ over $\Q$ given by an affine equation as described in \S\ref{sec:genus2curves}. We saw that the corresponding Jacobian can be embedded in $\P^8$. The identity element $\cO$ in $J$ satisfies $X_{111} \neq 0$. Let us look at the affine part of $J$ where $X_{111} \neq 0$, and dehomogenise the defining equations of $J$ accordingly. We write \[x = \frac{X}{X_{111}}, \quad x_0 = \frac{X_0}{X_{111}}, \quad x_{ij} = \frac{X_{ij}}{X_{111}}, \quad x_{ijk} = \frac{X_{ijk}}{X_{111}}.\] Consider the local ring $\cO_{J, \cO}$ of $J/\Q$ at $\cO$. The maximal ideal $\mathfrak{m}$ of $\cO_{J, \cO}$ is the ideal generated by the functions $x_0, x_{11}, x_{12}, x_{22}, x_{112}, x_{122}, x_{222}, x$.

\begin{lemma}[{\cite[Theorem 4.2]{grant}}]\label{lemma:generatorst1t2}
    Let $t_1 = - x_{11}$ and $t_2 = - x$. Then \[\widetilde{\cO}_{J, \cO} \cong \Q[[t_1, t_2]],\] where $\widetilde{\cO}_{J, \cO}$ denotes the completion of $\cO_{J, \cO}$ with respect to its maximal ideal. 
\end{lemma}

In particular, we can find expansions for the coordinate functions $x_0, x_{ij}, x_{ijk}$ as power series in $\Q[[t_1, t_2]]$. In \cite[Theorem 4.2]{grant} it is shown that these expansions have coefficients in $\Z$. We have for example
\begin{equation}\label{eq:coordinatefunctionsexpansions}
\begin{aligned}
    x_0 &= t_1^3\left(-1 -f_2 t_1^2 + \sum_{i, j \geq 0, \; i + j \geq 4} \alpha_{ij}t_1^i t_2^j\right),\\
    x_{22} &= t_1\left(-2t_1 t_2 + \sum_{i, j \geq 0,\; i + j \geq 3} \beta_{ij}t_1^i t_2^j\right),\\
    x_{12} &= t_1\left( t_2^2 + \sum_{i, j \geq 0,\; i + j \geq 3} \gamma_{ij}t_1^i t_2^j\right), 
\end{aligned}
\end{equation}
with $\alpha_{ij}, \beta_{ij}, \gamma_{ij} \in \Z$. The description of $x_0$ is a refinement of the description in \cite[Theorem 4.2]{grant} which we need in \S\ref{sec:comparison}, see \cite[Lemma 1.6.16]{tripThesis} for a derivation. It follows that
\begin{align*}
     x_0^{-1} &= -t_1^{-3} \left(1 - f_2 t_1^2 + \sum_{i, j \geq 0, \; i + j \geq 4} \alpha_{ij}'t_1^i t_2^j\right)
\end{align*}
in $\operatorname{Frac}(\Q[[t_1, t_2]])$, with $\alpha_{ij}' \in \Z$.
Using this and the expansions in \eqref{eq:coordinatefunctionsexpansions}, we find
\begin{equation}\label{eq:coordinatefunctionsKummersurface}
    \begin{aligned}
        \wp_{22} &= x_{22}/x_0 = t_1^{-2}\left(2 t_1 t_2 + \sum_{i, j \geq 0,\; i + j \geq 3} \beta_{ij}' t_1^i t_2^j\right),\\
        -\wp_{12} &= -x_{12}/x_0 = t_1^{-2}\left(t_2^2 + \sum_{i, j \geq 0,\; i + j \geq 3} \gamma_{ij}'t_1^i t_2^j\right),\\
        \wp_{11} &= x_{11}/x_0 = t_1^{-2} \left(1 - f_2 t_1^2 + \sum_{i, j \geq 0, \; i + j \geq 4} \alpha_{ij}'t_1^i t_2^j\right)
    \end{aligned}
\end{equation}
in $\operatorname{Frac}(\Q[[t_1, t_2]])$, where $\beta_{ij}', \gamma_{ij}' \in \Z$.

We denote by $J_1(\Q_p)$ the model-dependent kernel of reduction modulo $p$ of $J(\Q_p)$. That is, we consider the integral model of $J$ as embedded in $\P^8$ described above, which can be found in \cite[Corollary 2.15]{grant}, and reduce the equations modulo $p$. This gives us a (possibly singular) variety $\tilde{J}$ in $\P^8$ over $\F_p$. We have a natural reduction map $\tilde{} \, \colon J(\Q_p) \ra \tilde{J}(\F_p)$, and define \[J_1(\Q_p) = \{P \in J(\Q_p) \mid \tilde{P} = \tilde{\cO}\},\] which is a finite index subgroup of $J(\Q_p)$ (see \cite[Chapter 7, \S 5]{casselsflynn}, which uses \cite[III, \S 6]{mattuck}). Note that when $P \in J_1(\Q_p)$, we have $t_1(P), t_2(P) \in p\Z_p$. It turns out there is a bijection (\cite[Corollary 4.5]{grant})
\begin{align*}
    \bt \colon J_1(\Q_p) &\ra p\Z_p \times p\Z_p\nonumber\\
    P &\mapsto (t_1(P), t_2(P)).
\end{align*}
This bijection induces a formal group structure on $p\Z_p \times p\Z_p$ (see \cite[Theorem 4.6]{grant}). More explicitly, we can define a pair of power series $\boldsymbol{F}_J(\bX, \bY)$ which is the group law of a $2$-parameter formal group $\hat{J}$ over $\Z_p$, in such a way that for all $P, Q \in J_1(\Q_p)$, we have
\begin{align*}
    \bF_J(\bt(P), \bt(Q)) &= \bt(P + Q).
\end{align*}
In other words, $\bt$ is a group homomorphism from $J_1(\Q_p)$ to the group $\hat{J}\left(p\Z_p \times p\Z_p\right)$ associated to the formal group ($\hat{J}, \bF_J)$. 

We note that $t_1= -2\frac{\wp_{11}}{\wp_{111}}$ and $t_2 = -\frac{\wp + f_3 \wp_{12} - f_1}{\wp_{111}}$ are odd functions, in the sense that $t_i(-P) = - t_i(P)$. This follows from the fact that $\wp_{11}$, $\wp_{12}$, and $\wp$ are even functions and $\wp_{111}$ is an odd function on $J$. This implies that the inverse on $\hat{J}$ is simply $\boldsymbol{i}_J(\bT) = -\bT$, and $\bF_J(-\bX, -\bY) = -\bF_J(\bX, \bY)$. Hence $\bF_J$ only has terms of total odd degree. The following result follows from Lemma \ref{lem:multipleofpninformalgroup}.

\begin{corollary}\label{cor:ordtipnP}
    Let $P \in J_1(\Q_p)$. For all $n \geq 0$ and for $i = 1, 2$ we have 
    \begin{equation*} 
        \ord_p(t_i(p^n P)) \geq n + 1.
    \end{equation*}
\end{corollary}
Furthermore, it follows from the identification with a formal group that $J_1(\Q_p)$ has trivial torsion \cite[Theorem 7.4.1, Corollary]{casselsflynn}.

\begin{remark}
    Grant restricts to genus 2 curves that can be given by an equation of the form $y^2 = f(x)$ with $f$ of degree 5, as we do here. For genus 2 curves given by an equation with $f$ of degree 6, a similar description of the Jacobian and the corresponding formal group is given by Flynn in \cite{flynn1}.
\end{remark}

\subsection{Local real heights}\label{sec:localrealheights}
To prove our main results, some of our arguments make use of local real height functions on $J$ as defined in \cite{uchida}. Using the theta divisor, we define the divisors
\begin{align*}
    \Theta_1 = 2\Theta, \quad \Theta_2 = \Theta_1 + \divv{\wp_{22}}, \quad  \Theta_3 = \Theta_1 + \divv{\wp_{12}}, \quad \Theta_4 = \Theta_1 + \divv{\wp_{11}}.
\end{align*}
We have $P \in \supp(\Theta_i)$ precisely when the $i$-th projective coordinate of $\kappa(P)$ vanishes. For $D \in \Div(J)$, we write \[J_D \colonequals J \setminus \supp(D).\]

\begin{definition}[{\cite[119]{uchida}}]\label{def:naivereallocalheightJ}
    Let $q$ be a prime. For $i \in \{1, 2, 3, 4\}$, we define the \term{naive local real height} $\lambda_{i, q} \colon J_{\Theta_i}(\Q_q) \ra \R$ as follows. For $P \in J_{\Theta_i}(\Q_q)$, we write $\kappa(P) = [k_1 : \cdots : k_4]$. Then
    $$\lambda_{i, q}(P) = \log \max_{1 \leq j \leq 4} \left|\frac{k_j}{k_i}\right|_q.$$
\end{definition}
We note that this definition is independent of the choice of homogeneous coordinates for $\kappa(P)$. If $P \in J_{\Theta_i}(\Q)$, we have $$\lambda_{i, q}(P) = - \log \left|w_i(P)\right|_q$$
for all $1 \leq i \leq 4$.

\begin{definition}[{\cite[119]{uchida}}]\label{def:canonicallocalrealheight}
    Let $q$ be a prime. For $i \in \{1, 2, 3, 4\}$, we define the \term{canonical local real height} $\hat{\lambda}_{i, q} \colon J_{\Theta_i}(\Q_q) \ra \R$ by $$\hat{\lambda}_{i, q}(P) = \lambda_{i,q}(P) + \sum_{n = 0}^\infty \frac{1}{4^{n+1}} \log \Phi_q(2^n P),$$
    where, with $\kappa(P) = [k_1 : \cdots : k_4]$, we define \[\Phi_q(P) = \frac{\max_{1 \leq j \leq 4} |\delta_j(k_1, \ldots, k_4)|_q}{\max_{1 \leq j \leq 4}|k_j|_q^4}.\]
    This is independent of the choice of homogeneous coordinates for $\kappa(P)$.
\end{definition}

For each prime $q$, let us consider the following subset of $J(\Q_q)$: \[U_q \colonequals \{P \in J(\Q_q) \mid \Phi_q(P) = 1\}.\] 
If $P \in J_{\Theta_i}(\Q)$, we have
\begin{equation}\label{eq:Phiq}
    \Phi_q(P) = \max_{1 \leq j \leq 4} |\delta_j(w(P))|_q.
\end{equation}

\begin{proposition}[{\cite[Theorem 4.1]{stollheight}, \cite[Lemma 1]{flynnsmart}}]\label{prop:Usubgroup}
    Let $q$ be a prime. Then
    \begin{enumerate}
        \item[(a)] \label{item1:Usubgroup}$U_q$ is a subgroup of $J(\Q_q)$ of finite index, and $J_1(\Q_q) \subseteq U_q$.
        \item[(b)] If $q$ is odd and $J/\Q$ has good reduction at $q$, then $U_q = J(\Q_q)$.
    \end{enumerate}
\end{proposition}

\begin{corollary}\label{prop:lambdainUq}
    Let $P \in U_q \setminus \supp(\Theta_i)$ for some $i \in \{1, 2, 3, 4\}$. Then \[\hat{\lambda}_{i,q}(P) = \lambda_{i, q}(P).\] In particular, if $q$ is odd and $J$ has good reduction at $q$, this is true for all $P \in J_{\Theta_i}(\Q_q)$.
\end{corollary}

The following theorem shown by Uchida gives relations between the canonical local height functions evaluated on sums of points, and the biquadratic forms $B_{ii}$ and multiplication-by-$m$ functions $\mu_m$.

\begin{theorem}[{\cite[Theorem 5.3, Theorem 5.6]{uchida}}]\label{thm:uchidalocalheights}
    Let $q$ be a prime. For $i \in \{1, 2, 3, 4\}$, let $\hat{\lambda}_{i,q}$ be the canonical local real height defined in Definition \ref{def:canonicallocalrealheight}.
    \begin{enumerate}[label=(\roman*), ref= \ref{thm:uchidalocalheights}(\roman*)]
        \item \label{item:uchida1} For any $m \in \Z_{>0}$ and $P \in J(\Q)$ with $P, mP \notin    \supp(\Theta_i)$, we have
        \begin{align*}
            \hat{\lambda}_{i, q}(mP) - m^2 \hat{\lambda}_{i, q}(P) = - \log \left|\frac{\mu_{m, i}(w(P))}{w_i(P)^{m^2}}\right|_q.
        \end{align*}

        \item \label{item:uchida2} For any $P, Q\in J(\Q)$ such that $P, Q, P+Q, P-Q \notin \supp(\Theta_i)$, we have
    \begin{align*}
        \hat{\lambda}_{i, q}(P + Q) + \hat{\lambda}_{i, q}(P - Q) - 2\hat{\lambda}_{i, q}(P) - 2\hat{\lambda}_{i, q}(Q) &= - \log \left| \frac{B_{ii}(w(P), w(Q))}{w_i(P)^2 w_i(Q)^2}\right|_q.
    \end{align*}
    \end{enumerate}
\end{theorem}

\begin{remark}
    Uchida states this theorem more generally for $\Q_p$-rational points. We will only use it for points in $J(\Q)$, so we phrased it here in terms of the normalisations $w(P)$ and $w(Q)$ of the coordinates of $\kappa(P)$ and $\kappa(Q)$, respectively.
\end{remark}

\section{Naive \texorpdfstring{$p$}{p}-adic height function}\label{sec:naiveheight}
We define a naive $p$-adic height function on the following finite index subgroup of $J(\Q)$:
\begin{equation}
    J_p \colonequals J(\Q) \cap J_1(\Q_p) \cap \left( \bigcap_{q \; \text{prime}} U_q \right).\label{eq:defJp}
\end{equation}
We note that $J_p$ is a subgroup of finite index in $J(\Q)$ because it is a finite intersection of subgroups of finite index (by Proposition \ref{prop:Usubgroup}, $U_q \cap J(\Q) = J(\Q)$ for all but finitely many $q$).

\begin{definition}\label{def:Hp}
    Let $p$ be an odd prime, and let $J_p$ be as above. We define a naive $p$-adic height on $J_p$ as follows:
    \begin{equation}\label{eq:naivepadicheightJ}
        \begin{aligned}
            H_p \colon J_p &\rightarrow \Q_p\\
        H_p(P) &= \log_p(w_4(P)),
        \end{aligned}
    \end{equation}
    where $w_4(P)$ is the fourth coordinate of $\kappa(P)$ in the normalisation defined in \eqref{eq:normalizationkappa}.
\end{definition}
Note that $J_1(\Q_p) \subseteq J_{\Theta_4}(\Q_p)$, because for $P \in J_1(\Q_p)$, we have that $\kappa(P)$ reduces to $\kappa(\cO) = [0:0:0:1]$ modulo $p$, and hence $w_4(P) \neq 0$. This ensures that $H_p$ is well-defined.

\begin{remark}
    Note the similarity between the formula for $H_p$ in \eqref{eq:naivepadicheightJ} and the formula for the naive $p$-adic height on elliptic curves in \eqref{eq:naivepadicheightE} defined by Perrin-Riou. Because $\kappa_E(\cO) = [1 : 0]$, we have for $P \in E_p(\Q)$ that $w_1(P)$ is the coordinate of $\kappa_E(P)$ with the highest $p$-adic absolute value. Similarly, for $P \in J_p$, we have that $w_4(P)$ is the coordinate of $\kappa(P)$ with the highest $p$-adic absolute value. So both $H_2$ and $H_p$ are defined by normalising the Kummer coordinates of a point in such a way that they are coprime integers, and taking the $p$-adic logarithm of the coordinate with the highest $p$-adic absolute value.
    
    We can go on to compare with the standard naive real height, both on elliptic curves and the Jacobians of genus two curves. It is defined as 
    $H(P) = \log \max_i\, |w_i(P)|_\infty$, where $i$ runs over the Kummer coordinates of $P$ in the normalisation described above. Again we take the coordinate with the highest absolute value, in this case at the place $\infty$ (because we consider the real height). In both the real and $p$-adic case, it can be observed that the formula for the naive height arises from applying an id\`{e}le class character and taking the $p$ or $\infty$ component, respectively. In the $p$-adic case, we use the cyclotomic $p$-adic id\`{e}le class character whose component at $p$ is $x \mapsto \log_p(x)$ (see \cite[\S 2.1]{BBBM}). In the real case, we use the real-valued id\`{e}le class character whose component at $\infty$ is $x \mapsto \log |x|_\infty$.
\end{remark}

Our main result, Theorem \ref{thm:mainthm}, says that we can use the naive height $H_p$ to define a quadratic height on $J(\Q)$. We now prove the first part of this result.

\begin{theorem}\label{thm:existencelimit}
Let $P \in J_p$, with $J_p$ as in \eqref{eq:defJp}. Then the following limit exists:
\[h_p(P) \colonequals \lim_{n \ra \infty} \frac{1}{p^{2n}} H_p(p^n P).
\]
\end{theorem}

Note that the definition of $H_p$ involves the normalisation $w(P)$ of $\kappa(P)$. As a first step in proving Theorem \ref{thm:existencelimit}, we want to express $w(mP)$ in terms of $w(P)$ for $m \in \Z_{>0}$. It follows from \eqref{eq:Phiq} that for $P \in J_p$, we have $w_i(2P) = \pm \delta_i(w(P))$ for $i \in \{1, 2, 3, 4\}$. We use Theorem \ref{item:uchida1} to obtain an analogous result for arbitrary multiples of $P$.

\begin{proposition}\label{prop:muxcoordinates}
    Let $P \in J_p$. Then $w_i(mP) = \pm \mu_{m, i}(w(P))$ for $i \in \{1, 2, 3, 4\}$ and all $m \in \Z_{>0}$.
\end{proposition}

\begin{proof}
    Consider any prime $q$ and any $m \in \Z_{>0}$. Recall that because $P, mP \in J_p$, we have $w_4(P), w_4(mP) \neq 0$ and hence $P, mP \notin \supp(\Theta_4)$. We make use of the canonical local real height functions introduced in Definition \ref{def:canonicallocalrealheight}. We apply Theorem \ref{item:uchida1} to obtain
    \begin{align}
        \hat{\lambda}_{4, q}(m P) - m^2 \hat{\lambda}_{4, q}(P)
        &= - \log  \left| \mu_{m, 4}(w(P))\right|_q + m^2 \log \left| w_4(P) \right|_q. \label{eq:coprime1}
    \end{align}
    On the other hand, because $P, m P \in U_q$, we have from Proposition \ref{prop:lambdainUq} that
    \begin{equation}\label{eq:coprime2}
    \begin{aligned}
        \hat{\lambda}_{4, q}(m P) - m^2 \hat{\lambda}_{4, q}(P) &= \lambda_{4, q}(m P) - m^2 \lambda_{4, q}(P)\\
        &= - \log \left|w_4(m P)\right|_q + m^2 \log \left|w_4(P)\right|_q. 
    \end{aligned}
    \end{equation}
    Equating \eqref{eq:coprime1} and \eqref{eq:coprime2} then gives
    \begin{align*}
        - \log \left|w_4(m P)\right|_q &= - \log \left|\mu_{m, 4}(w(P))\right|_q.
    \end{align*}
    Hence, $\left|w_4(m P)\right|_q = \left|\mu_{m, 4}(w(P))\right|_q$ for all primes $q$. Because ${w_4(m P), \mu_{m, 4}(w(P)) \in \Z}$, this implies that $w_4(m P) = \pm \mu_{m, 4}(w(P))$. We have
    \begin{align*}
        \kappa(m P) &= [w_1(m P) : w_2(m P) : w_3(m P) : w_4(m P)]\\
        &= [\mu_{m, 1}(w(P)): \mu_{m, 2}(w(P)): \mu_{m, 3}(w(P)): \mu_{m, 4}(w(P))],  
    \end{align*}
    so this implies that $w_i(m P) = \pm \mu_{m, i}(w(P))$ for all $1 \leq i \leq 4$. 
\end{proof}

In particular, we deduce that \[\gcd(\mu_{m, 1}(w(P)), \mu_{m, 2}(w(P)), \mu_{m, 3}(w(P)), \mu_{m, 4}(w(P))) = 1\] for all $m \geq 0$. 

\begin{corollary}
    Let $P \in J_p$. Then \[H_p(mP) = \log_p(\mu_{m, 4}(w(P)))\] for all $m \geq 1$.
\end{corollary}
 
To prove Theorem \ref{thm:existencelimit}, we use that $\Q_p$ is a complete metric space with a non-archimedean absolute value, so that we can instead show that the sequence inside the limit is Cauchy by proving that
\begin{align}
    \lim_{n \ra \infty}  \frac{1}{p^{2(n+1)}} \left( H_p(p^{n+1}P) - p^2 H_p(p^nP)\right) = \lim_{n \ra \infty}  \frac{1}{p^{2(n+1)}} \log_p \left( \frac{\mu_{p, 4}(w(p^nP))}{w_4(p^nP)^{p^2}} \right) = 0.\label{eq:limitresult}
\end{align}

Note using \eqref{eq:kummermap} that for $Q \in J_{\Theta_4}(\Q)$, we have 
\begin{align*}
    \frac{\mu_{p, 4}(w(Q))}{w_4(Q)^{p^2}} = \mu_{p, 4}\left(\frac{1}{\wp_{11}}(Q), \frac{\wp_{22}}{\wp_{11}}(Q), -\frac{\wp_{12}}{\wp_{11}}(Q), 1\right).
\end{align*}

For any $m \geq 1$, we can consider $\mu_{m, 4}\left(\frac{1}{\wp_{11}}, \frac{\wp_{22}}{\wp_{11}}, -\frac{\wp_{12}}{\wp_{11}}, 1\right)$ as an element of the function field $\Q(J) \subseteq \operatorname{Frac}(\Q[[t_1, t_2]])$. Using \eqref{eq:coordinatefunctionsKummersurface}, we find
\begin{equation}\label{eq:fractionalexpansions}
    \begin{aligned}
        \frac{1}{\wp_{11}} &= t_1^2 + (\text{terms of total degree $\geq 3$})\\
        \frac{\wp_{22}}{\wp_{11}} &= 2 t_1 t_2 +  (\text{terms of total degree $\geq 3$})\\
        -\frac{\wp_{12}}{\wp_{11}} &= t_2^2 + (\text{terms of total degree $\geq 3$}) 
    \end{aligned}
\end{equation}

which are all in $\Z[[t_1, t_2]]$. To find properties of the expansion of $\mu_{m, 4}\left(\frac{1}{\wp_{11}}, \frac{\wp_{22}}{\wp_{11}}, -\frac{\wp_{12}}{\wp_{11}}, 1\right)$ in terms of $t_1$ and $t_2$, we look more closely at the coefficients of the polynomials $\mu_{m,i}$. The following observation is a refinement of \cite[Lemma 3.9]{uchida}.

\begin{lemma}\label{lemma:mucoef}
    For $m \geq 1$, we have \[\mu_{m, 4}(k_1, k_2, k_3, k_4) \in k_4^{m^2} + (k_1, k_2, k_3)^2\Z[k_1, k_2, k_3, k_4].\]
\end{lemma}

\begin{proof}
    From the definition in Theorem \ref{thm:definitionmu}, we know that $\mu_{1, 4} = k_4$ satisfies the statement. To show the statement for arbitrary $m \geq 1$, we use an inductive argument using the definition of $\mu_m$. From \cite[Lemma 3.9]{uchida}, we obtain that  
    \begin{equation}\label{eq:uchida3.9}
        \mu_{m, i}(k_1, k_2, k_3, k_4) \in (k_1, k_2, k_3) \Z[k_1, k_2, k_3, k_4] \ \text{for} \ i = 1, 2, 3 \ \text{and} \ m \geq 0.
    \end{equation}
     From the equations for $\delta$ in \cite[Appendix C]{flynn2}, we see that \[\delta_4 \in k_4^4 + (k_1, k_2, k_3)^2\Z[k_j].\] Using \eqref{eq:uchida3.9}, it follows that if $\mu_{m, 4}$ satisfies the statement of the lemma, then so does $\mu_{2m, 4} = \delta_4(\mu_m)$.
    Furthermore, we see from the equations for the biquadratic forms $B_{ij}$ in \cite{biquadratic} that \[B_{44} \in k_4^2l_4^2 + (k_1, k_2, k_3)(l_1, l_2, l_3)\Z[k_j, l_j].\] Again, using \eqref{eq:uchida3.9}, it follows that if $\mu_{m, 4}$ and $\mu_{m+1, 4}$ satisfy the statement of the lemma, then so does $\mu_{2m+1, 4}$. For more details, see \cite[Lemma 3.3.7, Lemma 3.3.8]{tripThesis}.
\end{proof}

This leads to the following observation.

\begin{proposition}\label{prop:formoffunction}
    Let $m \geq 1$. Then $\mu_{m, 4}\left(\frac{1}{\wp_{11}}, \frac{\wp_{22}}{\wp_{11}}, -\frac{\wp_{12}}{\wp_{11}}, 1\right)$ has an expansion in $t_1$ and $t_2$ of the form $u_m(t_1, t_2)$ for some \[u_m \in 1 + (t_1, t_2)^4\Z[[t_1, t_2]].\]
\end{proposition}

\begin{proof}
    Recall that $\mu_{m,4}$ is a homogeneous polynomial of degree $m^2$. We know from Lemma \ref{lemma:mucoef} that $\mu_{m, 4} \in k_4^{m^2} + (k_1, k_2, k_3)^2\Z[k_1, k_2, k_3, k_4]$. From the expansions in \eqref{eq:fractionalexpansions}, we note that $\frac{1}{\wp_{11}}, \frac{\wp_{22}}{\wp_{11}}, -\frac{\wp_{12}}{\wp_{11}} \in (t_1, t_2)^2\Z[[t_1, t_2]]$. This implies the result.
\end{proof}

For $Q \in J_1(\Q_p) \cap J(\Q)$, we then have
\begin{align}
    u_m(t_1(Q), t_2(Q)) &= \frac{\mu_{m, 4}(w(Q))}{w_4(Q)^{m^2}}.\label{eq:muexpansionatP}
\end{align}

To arrive at the limit result in \eqref{eq:limitresult}, we need a general convergence lemma, which is a generalisation of \cite[Lemme part 1]{perrinriou2}. 

\begin{lemma}\label{lemma:mostgenerallemma}
    Let $\bT = (T_1, \ldots, T_r)$ and $g(\bT) \in 1 + (T_1, \ldots, T_r)^k \Q_p[[\bT]]$ for some $k \in \Z_{>0}$, such that $g$ converges on some neighborhood of $\boldsymbol{0} = (0,\ldots, 0)$ in $(\Q_p)^r$. Let $(\boldsymbol{x}^{(n)
    })$ be a sequence of $r$-tuples in $(\Q_p)^r$, satisfying $\ord_p(x_i^{(n)}) \geq n$ for all $n \geq 0$ and $i = 1, \ldots, r$. Then for large enough $n \in \Z_{\geq 0}$, $g(\boldsymbol{x}^{(n)})$ converges, and for $m \in \Z_{<k}$ we have
    \begin{align*}
        \lim_{n \ra \infty} \frac{1}{p^{mn}}\log_p(g(\boldsymbol{x}^{(n)})) = 0.
    \end{align*}
\end{lemma}

\begin{proof}
    We know that $g$ converges on a neighborhood of $\boldsymbol{0}$, so let $R \in \Z$ be such that $g(\boldsymbol{x})$ converges for all $\boldsymbol{x} \in (p^R\Z_p)^r$. Then $g(\boldsymbol{x}^{(n)})$ converges for $n \geq R$.
    
    Let us write $$g(\bT) = 1 + \sum_{\substack{i_1, \ldots, i_j \geq 0\\ i_1 + \cdots + i_r \geq k}}^\infty a_{i_1, \ldots, i_r} T_1^{i_1}\cdots T_r^{i_r}$$ with $a_{i_1, \ldots, i_r} \in \Q_p$. 
    There is an integer $M \in \Z_{\geq k}$ such that for all $i_1, \ldots, i_r \in \Z_{\geq 0}$ with $i_1 + \cdots + i_r \geq M$, we have $\ord_p(a_{i_1, \ldots, i_r} p^{(i_1 + \cdots + i_r)R}) \geq 0$, and hence
    \begin{align*}
        \ord_p(a_{i_1, \ldots, i_r}) \geq - kR.  
    \end{align*}

    Now let $n \geq R$. If $i_1 + \cdots + i_r \geq M$, then the corresponding term in the expansion of $g(\boldsymbol{x}^{(n)})$ satisfies
    \begin{align*}
        \ord_p\left(a_{i_1, \ldots, i_r} (x_1^{(n)})^{i_1} \cdots (x_r^{(n)})^{i_r}\right) &= \ord_p(a_{i_1, \ldots, i_r}) + \sum_{j = 1}^r i_j \, \ord_p(x_{j}^{(n)})\\
        &\geq k(n-R).
    \end{align*}

    If $k \leq i_1 + \cdots + i_r < M$, we obtain
    \begin{align*}
        \ord_p\left(a_{i_1, \ldots, i_r} (x_1^{(n)})^{i_1} \cdots (x_r^{(n)})^{i_r}\right) &= \ord_p(a_{i_1, \ldots, i_r}) + \sum_{j = 1}^r i_j \, \ord_p(x_{j}^{(n)})\\
        &\geq \ord_p(a_{i_1, \ldots, i_r}) + kn.
    \end{align*}

    We deduce that 
    \begin{align*}
        \ord_p(g(\boldsymbol{x}^{(n)})-1) &\geq \min_{\substack{i_1, \ldots, i_j \geq 0\\ i_1 + \cdots + i_r \geq k}} \ord_p\left(a_{i_1, \ldots, i_r} (x_1^{(n)})^{i_1} \cdots (x_r^{(n)})^{i_r}\right)\\
        &\geq kn + c,
    \end{align*}
    where $c$ is the minimum of $-kR$ and all $\ord_p(a_{i_1, \ldots, i_r})$ for $i_j$ such that $k \leq i_1 + \cdots + i_r < M$. Hence $g(\boldsymbol{x}^{(n)}) \in  1 + p^{kn + c} \Z_p$. We conclude that 
    \begin{align*}
        \ord_p\left(\frac{1}{p^{mn}}\log_p(g(\boldsymbol{x}^{(n)}))\right) &= -mn +  \ord_p(\log_p(g(\boldsymbol{x}^{(n)})))\\
        &\geq -mn + kn + c &\text{(using \eqref{eq:logarithm})}\\
        &= n(k-m) + c.
    \end{align*}
    Because $m < k$, this shows that $\ord_p\left(\frac{1}{p^{mn}}\log_p(g(\boldsymbol{x}^{(n)}))\right)$ approaches infinity as ${n \ra \infty}$.
\end{proof}

\begin{proof}[Proof of Theorem \ref{thm:existencelimit}]
    Recall that it suffices to show that the limit in \eqref{eq:limitresult} vanishes. Note that $p^nP \in J_p$ for all $n \geq 0$. Using \eqref{eq:muexpansionatP} we obtain 
    \begin{align}
         \lim_{n \ra \infty}  \frac{1}{p^{2(n +1)}} \log_p\left(\frac{\mu_{p, 4}(w(p^nP))}{w_4(p^nP)^{p^2}}\right) &= \frac{1}{p^2}\lim_{n \ra \infty}  \frac{1}{p^{2n}} \log_p\left(u_p(t_1(p^n P), t_2(p^n P))\right).\label{eq:limitexistence}
    \end{align}
    We have $u_p \in 1 + (t_1, t_2)^4\Z[[t_1, t_2]]$ by Proposition \ref{prop:formoffunction}, so in particular it converges when evaluated at points $Q$ for which $(t_1(Q), t_2(Q)) \in p\Z_p \times p\Z_p$. Because $P \in J_p \subseteq J_1(\Q_p)$, we have that $\ord_p(t_i(p^n P)) \geq n + 1$ for $i = 1, 2$ and for all $n \geq 0$ by Corollary \ref{cor:ordtipnP}. Then Lemma \ref{lemma:mostgenerallemma} implies that the limit in \eqref{eq:limitexistence} is zero.
\end{proof}

\begin{remark}
    Lemma \ref{lemma:mostgenerallemma} is a generalisation of \cite[Lemme part 1]{perrinriou2} (which covers the one-parameter case) in the sense that it is valid for a broader set of series: Perrin-Riou restricts to $g(T) \in 1 + \frac{T^3}{3}\Z_p[[\frac{T}{3}]]$. In \cite[Th\'{e}or\`{e}me]{perrinriou2}, she uses this lemma to show a relation between the $p$-adic height on elliptic curves which she constructs from a naive height in \cite{perrinriou}, and a $p$-adic height constructed using local heights and a $p$-adic sigma function, under the assumption that the curve $E$ has good ordinary reduction at $p$. In \cite[Theorem 2.3.23]{tripThesis}, we use Lemma \ref{lemma:mostgenerallemma} to extend her argument to the case where $E$ does not have ordinary reduction at $p$.
\end{remark}

\section{Quadraticity of \texorpdfstring{$h_p$}{hp}}\label{sec:quadraticity}

Next, we show that the function $h_p$ is quadratic on $J_p$. Explicitly we show the following result.

\begin{theorem}\label{thm:parallelogramlawhpJ}
    Let $P, Q \in J_p$, with $J_p$ as in \eqref{eq:defJp}. Then
    \begin{equation}\label{eq:quadraticform}
        h_p(P + Q) + h_p(P - Q) = 2 h_p(P) + 2 h_p(Q).
    \end{equation}
\end{theorem}

We first note that by definition, for $P, Q \in J_p$ we have
\begin{align}
    H_p(P+Q) + H_p(P-Q) &= \log_p(w_4(P+Q)w_4(P-Q)).\label{eq:parallellogram1}
\end{align}
Our goal is to rewrite the right-hand side in such a way that it depends on the Kummer coordinates of $P$ and $Q$ instead. As in the previous section, we use theory of the canonical local real height functions due to Uchida, which were introduced in \S\ref{sec:localrealheights}. We obtain a result similar to Proposition \ref{prop:muxcoordinates}.

\begin{proposition}\label{prop:Biixcoordinates}
    Let $P, Q \in J_p$. Then \[B_{ii}(w(P), w(Q)) = \pm w_i(P + Q) w_i(P-Q)\] for $i = 1, 2, 3, 4$.
\end{proposition}

\begin{proof}
    The proof is completely analogous to the proof of Proposition \ref{prop:muxcoordinates}, using Theorem \ref{item:uchida2}.
\end{proof}

Together, \eqref{eq:parallellogram1} and Proposition \ref{prop:Biixcoordinates} give
\begin{align}
    H_p(P+Q) + H_p(P-Q) &= \log_p(B_{44}(w(P), w(Q)))\label{eq:parallellogram2}
\end{align}
for $P, Q \in J_p$. Recall that $J_p \subseteq J_1(\Q_p) \subseteq J_{\Theta_4}(\Q_p)$. Using \eqref{eq:kummermap}, we then get that 
\begin{align*}
    B_{44}(w(&P), w(Q)) =\\
    &w_4(P)^2 w_4(Q)^2 B_{44}\left(\frac{1}{\wp_{11}}(P), \frac{\wp_{22}}{\wp_{11}}(P), -\frac{\wp_{12}}{\wp_{11}}(P), 1, \frac{1}{\wp_{11}}(Q), \frac{\wp_{22}}{\wp_{11}}(Q), -\frac{\wp_{12}}{\wp_{11}}(Q), 1\right).
\end{align*}
Consider the power series
\begin{align*}
    &\zeta(T_1, T_2, S_1, S_2) \colonequals\\
    &B_{44}\left(\frac{1}{\wp_{11}}(T_1, T_2), \frac{\wp_{22}}{\wp_{11}}(T_1, T_2), -\frac{\wp_{12}}{\wp_{11}}(T_1, T_2), 1, \frac{1}{\wp_{11}}(S_1, S_2), \frac{\wp_{22}}{\wp_{11}}(S_1, S_2), -\frac{\wp_{12}}{\wp_{11}}(S_1, S_2), 1\right)
\end{align*}
as an element of $\Z[[T_1, T_2, S_1, S_2]]$ using the expansions in \eqref{eq:fractionalexpansions}. Then $$B_{44}(w(P), w(Q)) = w_4(P)^2 w_4(Q)^2\zeta(\boldsymbol{t}(P), \boldsymbol{t}(Q)).$$
Recall from the proof of Lemma \ref{lemma:mucoef} that $B_{44} \in k_4^2l_4^2 + (k_1, k_2, k_3)(l_1, l_2, l_3)\Z[k_i, l_i]$, which implies that the constant coefficient of $\zeta$ is 1. Using that \[\frac{1}{\wp_{11}}, \frac{\wp_{22}}{\wp_{11}}, -\frac{\wp_{12}}{\wp_{11}} \in (t_1, t_2)^2\Z[[t_1, t_2]]\] (see \eqref{eq:fractionalexpansions}), we conclude that 
\begin{equation}
    \zeta(T_1, T_2, S_1, S_2) \in 1 + (T_1, T_2)^2(S_1, S_2)^2 \Z[[T_1, T_2, S_1, S_2]].\label{eq:zetaexpansion}
\end{equation}

\begin{proof}[Proof of Theorem \ref{thm:parallelogramlawhpJ}]
    Let $P, Q \in J_p$. We rewrite \eqref{eq:parallellogram2} as
    \begin{align*}
        H_p(P+Q) + H_p(P-Q) &= \log_p\left(w_4(P)^2 w_4(Q)^2 \zeta(\bt(P), \bt(Q))\right)\\
        &= 2 H_p(P) + 2 H_p(Q) + \log_p(\zeta(\bt(P), \bt(Q))).
    \end{align*}
    We obtain 
    \begin{align*}
        h_p(P + Q) + h_p(P - Q) &= \lim_{n \ra \infty} \frac{1}{p^{2n}} \left(H_p(p^n P + p^n Q) +  H_p(p^n P - p^nQ)\right)\\
        &= \lim_{n \ra \infty} \frac{1}{p^{2n}} \left(2 H_p(p^n P) + 2 H_p(p^n Q) + \log_p(\zeta(\bt(p^n P), \bt(p^n Q)))\right)\\
        &= 2 h_p(P) + 2 h_p(Q) + \lim_{n \ra \infty} \frac{1}{p^{2n}} \log_p(\zeta(\bt(p^n P), \bt(p^n Q))).
    \end{align*}
    Using \eqref{eq:zetaexpansion} and Corollary \ref{cor:ordtipnP}, it follows from Lemma \ref{lemma:mostgenerallemma} that the limit on the right is equal to zero.
\end{proof}

A straightforward consequence of Theorem \ref{thm:parallelogramlawhpJ} is the following.
\begin{corollary}\label{cor:hpquadraticfunction}
    Let $P \in J_p$ and let $n \in \Z$. Then 
    \begin{align*}
        h_p(n P) = n^2 h_p(P).
    \end{align*}
\end{corollary}

Using this property we can extend our definition to all of $J(\Q)$. 

\begin{definition}\label{def:hpJQ}
    Let $P \in J(\Q)$. Let $m \in \Z_{>0}$ such that $m P \in J_p$ (which exists because $J_p$ is a finite index subgroup of $J(\Q)$). Then we define
    \begin{align*}
        h_p(P) \colonequals \frac{1}{m^2}h_p(m P).
    \end{align*}
\end{definition}
Because $h_p$ is quadratic on $J_p$ by Corollary \ref{cor:hpquadraticfunction}, this definition is independent of the choice of $m$ and hence $h_p$ is well-defined on $J(\Q)$. It is the unique function that extends $h_p$ to $J(\Q)$ satisfying the property in Corollary \ref{cor:hpquadraticfunction} for all $P \in J(\Q)$. It follows straightforwardly that $h_p$ satisfies \eqref{eq:quadraticform} for all $P, Q \in J(\Q)$, and hence is a quadratic form on $J(\Q)$, which concludes the proof of Theorem \ref{thm:mainthm}.

\section{Comparison with Bianchi's quadratic height}\label{sec:comparison}

In \cite{bianchi3}, Bianchi defines global $p$-adic heights as a sum over all non-archimedean places $q \in M_\Q$ of $\Q$ of local $p$-adic heights \[\lambda_{J, q}^{(p)} \colon J_\Theta(\Q_q) \ra \Q_p\] (see \cite[Definitions 4.2, 4.5 and 4.8]{bianchi3}). We will show that the height we constructed in Definition \ref{def:hpJQ} is the same as one she defines.

\begin{remark}\label{rem:padicsigmafunctions}
The local $p$-adic height function at the prime $p$ is defined using a $p$-adic sigma function, named in analogy to the complex sigma function. When $C$ is a complex curve, there is a lattice $L \subset \C^2$ such that there exists an isomorphism $\Phi \colon {J \ra \C^2/L}$. There is a complex sigma function $\sigma(z_1, z_2)$ on $\C^2/L$ such that if we set \[\tilde{\wp}_{ij} = - \frac{\del^2}{\del z_i \del z_j} \log \sigma(z_1, z_2), \quad \tilde{\wp}_{ijk} = - \frac{\del^3}{\del z_i \del z_j \del z_k} \log \sigma(z_1, z_2) \quad \text{and} \quad \tilde{\wp} = \tilde{\wp}_{11} \tilde{\wp}_{22} - \tilde{\wp}_{12}^2,\] then on $J$ we obtain \[\tilde{\wp}_{ij} \circ \Phi = \wp_{ij}, \quad \tilde{\wp}_{ijk} \circ \Phi = \wp_{ijk} \quad \text{and} \quad \tilde{\wp} \circ \Phi = \wp\] (see \cite[38, 96]{baker}). This is the reason for the numbering of the functions $\wp_{ij}$ and $\wp_{ijk}$.  Similarly, when $C$ is a curve defined over $\Q_p$, there are differential operators $D_1$ and $D_2$ on $J$ \cite[Equation (2.3)]{bianchi3}. For every $\boldsymbol{c} =   (c_{ij})_{i, j} \in M_{2 \times 2}(\Q_p)$ with $c_{12} = c_{21}$, there is a unique odd power series $\sigma_p^{(\boldsymbol{c})}(t_1, t_2) \in t_1\left(1 + (t_1, t_2)\Q_p[[t_1, t_2]]\right)$ such that $- D_i D_j (\log \sigma_p^{(\boldsymbol{c})}(t_1, t_2)) = \wp_{ij} - c_{ij}$ for all $1 \leq i \leq j \leq 2$, where we consider $\wp_{ij}$ in terms of its power series expansion in \eqref{eq:coordinatefunctionsKummersurface} \cite[\S 3.2]{bianchi3}. These are called $p$-adic sigma functions because they satisfy analogous differential equations to the complex sigma function. Blakestad \cite[Proposition 34, Corollary 37]{blakestad} showed that there is a unique choice for $\boldsymbol{c}$ such that $\sigma_p^{(\boldsymbol{c})}(t_1, t_2)$ has coefficients in $\Z_p$. For this choice of $\boldsymbol{c}$, we call $\sigma_p^{(\boldsymbol{c})}$ the canonical $p$-adic sigma function. For $\boldsymbol{c} = \boldsymbol{0}$, the zero matrix, the expansion of the power series $\sigma_p^{(\boldsymbol{0})}(t_1, t_2)$ is directly derived from the expansion of the complex sigma function $\sigma$ via the formal logarithm, and it has coefficients in $\Q$ \cite[Theorem 3.3]{bianchi3}. The function $\sigma_p^{(\boldsymbol{0})}$ is called the naive $p$-adic sigma function.
\end{remark}

For each $p$-adic sigma function, Bianchi defines a local $p$-adic height at $p$, and each of these can be used to define a global $p$-adic height function. In this paper, we show that $h_p$ is equal to the height $h_p^{(\boldsymbol{0})}$ obtained by using $\sigma_p^{(\boldsymbol{0})}$, the naive $p$-adic sigma function. This height $h_p^{(\boldsymbol{0})}$ is quadratic (see \cite[24]{bianchi3}). Therefore, it suffices to compare it with $h_p$ on a finite index subgroup of $J(\Q)$.

Let $q_1, \ldots, q_r$ be the primes of bad reduction for our model of $J$. We define the subgroup
\begin{align*}
    J_{\operatorname{nice}} \colonequals J(\Q) \cap J_1(\Q_p) \cap \bigcap_{i = 1}^r J_1(\Q_{q_i})
\end{align*}
of $J(\Q)$. It follows from Proposition \ref{prop:Usubgroup} that $J_{\operatorname{nice}} \subseteq J_p$. Furthermore, $J_{\operatorname{nice}}$ is a subgroup of finite index in $J(\Q)$. We use this subgroup rather than $J_p$ because on $J_{\operatorname{nice}}$, there is a more straightforward description of the local heights of Bianchi, as we see in Proposition \ref{prop:standardformhhat}. We will show that $h_p$ and $h_p^{(\boldsymbol{0})}$ agree on $J_{\operatorname{nice}} \setminus \supp(\Theta)$, and then it will follow by the quadraticity of both heights that they agree on all of $J(\Q)$. On this restricted domain we get the following explicit descriptions.

\begin{proposition}\label{prop:standardformhhat}
    Let $P \in J_{\operatorname{nice}} \setminus \supp(\Theta)$. Then
    \begin{align*}
        \lambda_{J, p}^{(p)}(P) = - 2 \log_p(\sigma_p^{(\boldsymbol{0})}(\bt(P)))
    \end{align*}
    and for all primes $q \neq p$,
    \begin{align*}
        \lambda_{J, q}^{(p)}(P) = - \log_p |w_1(P)|_q.
    \end{align*}
    It follows that
    \begin{align*}
        h_p^{(\boldsymbol{0})}(P) = - \log_p \left(\frac{\sigma_p^{(\boldsymbol{0})}(\bt(P))^2}{w_1(P)} \right).
    \end{align*}
\end{proposition}

\begin{proof}
    The formula for $\lambda_{J, p}^{(p)}(P)$ follows directly from \cite[Definition 4.2]{bianchi3} taking ${m = 1}$. For $q \neq p$, it follows from \cite[Lemma 4.3 and Remark 4.6(ii)]{bianchi3} that we have \[\lambda_{J, q}^{(p)}(P) = -\log_p \left|\max_{i, j}\{|\wp_{ij}(P)|_q,1\}\right|_q.\] We have \[\kappa(P) = [1 : \wp_{22}(P) : -\wp_{12}(P) : \wp_{11}(P)] = [w_1(P) : w_2(P) : w_3(P) : w_4(P)].\] The normalisation of $w(P)$ requires that $$1 = \max_{1 \leq i \leq 4}|w_i(P)|_q = |w_1(P)|_q \max_{i, j} \{|\wp_{ij}(P)|_q, 1\}.$$ It follows that
    \begin{align*}
        \left|\max_{i, j}\{|\wp_{ij}(P)|_q,1\}\right|_q = \left(\max_{i, j}\{|\wp_{ij}(P)|_q, 1\}\right)^{-1} = |w_1(P)|_q.
    \end{align*}
    We obtain
    \begin{align*}
        h_p^{(\boldsymbol{0})}(P) &= -2 \log_p(\sigma_p^{(\boldsymbol{0})}(\bt(P))) + \sum_{q \neq p} - \log_p |w_1(P)|_q\\
        &= - \log_p \left( \sigma_p^{(\boldsymbol{0})}(\bt(P))^2 \prod_{q \neq p} |w_1(P)|_q\right)\\
        &= - \log_p \left(\frac{\sigma_p^{(\boldsymbol{0})}(\bt(P))^2}{w_1(P)} \right),
    \end{align*}
    where the last equality follows from the product formula.
\end{proof}

To show that $h_p = h_p^{(\boldsymbol{0})}$, we use a similar approach as Perrin-Riou in the case of elliptic curves in \cite{perrinriou2}. First, we find an expression $g(P)$ such that \[h_p^{(\boldsymbol{0})}(P) = H_p(P) + g(P)\] for all $P \in J_{\operatorname{nice}}\setminus \supp(\Theta)$. Then we want to take a limit over $p$-power multiples of $P$ on both sides such that the term involving $g$ vanishes, and we obtain $h_p(P)$ on the right-hand side. However, we only have this expression for points in $J_{\operatorname{nice}}\setminus \supp(\Theta)$. We start by noting that there are only finitely many multiples $p^n P$ that are in $\supp(\Theta)$, so that we can take the limit over the remaining multiples.

\begin{lemma}\label{lem:S(P)infinite}
    Let $P \in \left(J_1(\Q_p) \cap J(\Q)\right) \setminus \{\cO\}$. Then $p^n P \in \supp(\Theta)$ for only finitely many $n \in \Z_{>0}$.
\end{lemma}

\begin{proof}
    Suppose $p^n P \in \supp(\Theta)$ for infinitely many $n$. By definition of $\Theta$, $J(\Q) \cap \supp(\Theta)$ is precisely the set of rational points of $J$ contained in the image of $C$ under the Abel--Jacobi map with respect to $\infty$, which is in bijection with $C(\Q)$. But $C(\Q)$ is a finite set by Faltings's theorem (\cite{Faltings1}, \cite{Faltings2}). It follows that for some $m \neq n$, we have $p^n P = p^m P$. Hence $P$ is a torsion point, but this is not possible because $J_1(\Q_p)$ has trivial torsion.
\end{proof}

\begin{remark}\label{rem:alternativeproofLemma}
    Alternatively, we can show for $P \in J_1(\Q_p) \setminus \{\cO\}$ that $p^n P \notin \supp(\Theta)$ for infinitely many $n \in \Z_{>0}$ (which is enough for our purpose) using a geometric argument. Namely, we first show that $\Theta \cap [p]^*\Theta$ is finite. To see this, let $Q \in \supp(\Theta)$ such that $Q \neq \cO$. Then $Q$ is the image of some point $(x, y) \in C$ under the Abel--Jacobi map with respect to the point $\infty$. It is shown in \cite{cantor} that if $y \neq 0$ (so if $(x, y)$ is not one of the Weierstrass points of $C$), then $p Q \in \supp(\Theta)$ if and only if $P_p(x) = 0$, where $P_p$ is a nonzero polynomial in one variable with coefficients in $\Q$. There are only a finite number of points $(x, y) \in C$ such that $P_p(x) = 0$, so we conclude that \[S \colonequals \{Q \in \supp(\Theta) \mid p Q \in \supp(\Theta)\}\] is a finite set. Now suppose that $p^n P \notin \supp(\Theta)$ for only finitely many $n \in \Z_{>0}$. Then there exists an integer $m$ such that $p^i P \in \supp(\Theta)$ for all $i \geq m$, and hence $p^i P \in S$ for all $i \geq m$. Because $S$ is a finite set, this implies that $p^i P = p^j P$ for some $i \neq j$, and hence $P$ is a torsion point. But $J_1(\Q_p)$ has trivial torsion, so this contradicts the assumption that $P \neq \cO$.
\end{remark}

\begin{remark}
    In \cite[Theorem 3.3.22]{tripThesis}, we show that $J_1(\Q_p) \cap S = \{\cO\}$ using properties of the multiplication-by-$p$ map on the formal group associated to $J$, and use this as an alternative argument to show that for $P \in J_1(\Q_p) \setminus \{\cO\}$, we have $p^n P \notin \supp(\Theta)$ for infinitely many $n \in \Z_{>0}$.
\end{remark}

We now use Lemma \ref{lem:S(P)infinite} to show that $h_p$ and $h_p^{(\boldsymbol{0})}$ are equal.

\begin{theorem}\label{thm:comparisonhhatphpsubgroup}
    Let $P \in J_{\operatorname{nice}} \setminus \supp(\Theta)$. Then $$h_p^{(\boldsymbol{0})}(P) = h_p(P).$$ 
\end{theorem}

\begin{proof}
    Let us consider a point $Q \in J_{\operatorname{nice}} \setminus \supp(\Theta)$. Then $Q \in J_1(\Q_p)$ and $Q \neq \cO$, so $Q$ is not a torsion point, and $w_4(Q) \neq 0$. We get
    \begin{equation}\label{eq:expansionhphatusingHp}
         \begin{aligned}
        h_p^{(\boldsymbol{0})}(Q) &= - \log_p\left(\frac{\sigma_p^{(\boldsymbol{0})}(\bt(Q))^2}{w_1(Q)}\right) &\text{(Proposition \ref{prop:standardformhhat})} \\
        &=   - \log_p\left(\sigma_p^{(\boldsymbol{0})}(\bt(Q))^2 \frac{w_4(Q)}{w_1(Q)}\right) + \log_p(w_4(Q))\\
        &= - \log_p\left(\sigma_p^{(\boldsymbol{0})}(\bt(Q))^2 \wp_{11}(Q)\right) + H_p(Q).
    \end{aligned}
    \end{equation}
   
    Recall that $\wp_{11}$ has an expansion of the form \eqref{eq:coordinatefunctionsKummersurface}. We see in \cite[Appendix C]{bianchi3} that $\sigma_p^{(\boldsymbol{0})}$ has an expansion of the form
    \begin{align*}
        \sigma_p^{(\boldsymbol{0})}(\bt) \in t_1\left(1 + \frac{f_2}{2}t_1^2 + (t_1, t_2)^4 \Q[[t_1, t_2]]\right).
    \end{align*}
We deduce that
    \begin{align}
        \sigma_p^{(\boldsymbol{0})}(\bt)^2\wp_{11}(\bt) \in 1 + (t_1, t_2)^4\Q[[t_1, t_2]].\label{eq:expsigmap11}
    \end{align}
    Because $\sigma_p^{(\boldsymbol{0})}(\bt)$ converges for $\bt \in p\Z_p \times p\Z_p$ by \cite[Theorem 3.3]{bianchi3}, and $t_1^2\wp_{11}(\bt)$ converges on $p\Z_p \times p\Z_p$, we conclude that $\sigma_p^{(\boldsymbol{0})}(\bt)^2\wp_{11}(\bt)$ converges on a neighborhood of $(0, 0)$.
    
    Using Lemma \ref{lem:S(P)infinite}, we get
    \begin{align*}
        h_p^{(\boldsymbol{0})}(P) &= \lim_{n \ra \infty} \frac{1}{p^{2n}} h_p^{(\boldsymbol{0})}(p^n P)\\
        &= \lim_{n \ra \infty} \frac{1}{p^{2n}} H_p(p^n P) - \lim_{n \ra \infty} \frac{1}{p^{2n}} \log_p\left(\sigma_p^{(\boldsymbol{0})}(\bt(p^n P))^2 \wp_{11}(p^n P)\right) \qquad \ \ \text{(using \eqref{eq:expansionhphatusingHp})}\\
        &= h_p(P). \qquad \qquad \qquad \qquad \qquad \qquad \; \, \, \text{(using  \eqref{eq:expsigmap11}, Corollary \ref{cor:ordtipnP} and Lemma \ref{lemma:mostgenerallemma})}
    \end{align*}
\end{proof}

\begin{proof}[Proof of Theorem \ref{thm:comparison}]
    It follows from the quadraticity of $h_p$ and $h_p^{(\boldsymbol{0})}$ that $h_p = h_p^{(\boldsymbol{0})}$ on all of $J(\Q)$. Namely, if $P \in J(\Q) \setminus J_{\operatorname{tors}}$, we have by \cite[Lemma 4.1]{bianchi3} that there is a positive integer $m$ such that \[mP \in J_{\operatorname{nice}} \setminus \supp(\Theta),\] and then $h_p(P) = \frac{1}{m^2} h_p(m P) = \frac{1}{m^2} h_p^{(\boldsymbol{0})}(m P) = h_p^{(\boldsymbol{0})}(P)$. For $P \in J_{\operatorname{tors}}(\Q)$ we have $h_p(P) = h_p^{(\boldsymbol{0})}(P) = 0$.
\end{proof}

In this paper, we gave an alternative construction of the height $h_p^{(\boldsymbol{0})}$ in terms of a limit, and showed the existence of the limit and quadraticity of the resulting height independently of the construction of Bianchi. As suggested by the referee, we can adapt the definition of the naive height to recover the heights $h_p^{(\boldsymbol{c})}$ for other choices of $\boldsymbol{c} = (c_{ij})_{i, j} \in M_{2 \times 2}(\Q_p)$ with $c_{12} = c_{21}$. Using the existence of Bianchi's heights $h_p^{(\boldsymbol{c})}$, we can show that these naive heights $H_p^{(\boldsymbol{c})}$ in a limit recover those quadratic heights. 

Fix some $\boldsymbol{c} = (c_{ij})_{i, j} \in M_{2 \times 2}(\Q_p)$ with $c_{12} = c_{21}$. Let us write \[\wp^{(\boldsymbol{c})} = \wp_{11} + c_{22} \wp_{12} - c_{12} \wp_{22} - c_{11}.\] It follows from \eqref{eq:coordinatefunctionsKummersurface} that \[\wp^{(\boldsymbol{c})} \in t_1^{-2} \left(1 - f_2 t_1^2 - c_{22}t_2^2 -2c_{12} t_1 t_2 - c_{11}t_1^2 + (t_1, t_2)^3\Q_p[[t_1, t_2]] \right).\]
From \cite[\S 3.2, Appendix C]{bianchi3}, we know that
\[\sigma_p^{(\boldsymbol{c})}(\boldsymbol{t})^2 \in t_1^2 \left(1 + (f_2 + c_{11}) t_1^2 + 2 c_{12} t_1 t_2 + c_{22} t_2^2 + (t_1, t_2)^4 \Q_p[[t_1, t_2]] \right),\]
so we obtain
\[\sigma_p^{(\boldsymbol{c})}(\boldsymbol{t})^2 \wp^{(\boldsymbol{c})} \in 1 + (t_1, t_2)^3\Q_p[[t_1, t_2]].\]
By \cite[Proposition 3.4(i)]{bianchi3}, $\sigma_p^{(\boldsymbol{c})}$ converges on $p^m \Z_p \times p^m \Z_p$ for some $m \in \Z_{>0}$. We also see that $t_1^2 \wp^{(\boldsymbol{c})}$ converges on $p\Z_p \times p\Z_p$, so we conclude that $\sigma_p^{(\boldsymbol{c})}(\boldsymbol{t})^2 \wp^{(\boldsymbol{c})}$ converges on a neighborhood of $(0,0)$.

Let us define a new subgroup \[J_{\boldsymbol{c}} = \{P \in J_{\operatorname{nice}} \mid \boldsymbol{t}(P) \in p^d \Z_p \times p^d \Z_p\} \subseteq J(\Q),\] where $d = \max_{i, j}\left\{\left\lceil\frac{1}{2}(1-\ord_p(c_{ij}))\right\rceil, m\right\}$. It follows from the argument in \cite[Chapter 7, \S 5]{casselsflynn} that this is a finite index subgroup of $J(\Q)$. We define more generally the naive $p$-adic height\[H_p^{(\boldsymbol{c})}(P) \colonequals \log_p(w_4(P)-c_{22}w_3(P) - c_{12} w_2(P) - c_{11} w_1(P))\] for $P \in J_{\boldsymbol{c}} \setminus \supp(\Theta)$. We note that $H_p^{(\boldsymbol{c})}(P)$ is well-defined on its domain of definition. Namely, \[w_4(P)-c_{22}w_3(P) - c_{12} w_2(P) - c_{11} w_1(P) = w_1(P) \wp^{(\boldsymbol{c})}(\boldsymbol{t}(P)),\] and $w_1(P) = 0$ if and only if $\frac{1}{\wp_{11}}(\boldsymbol{t}(P)) = 0$. We can deduce from \eqref{eq:coordinatefunctionsKummersurface} that this happens precisely when $t_1(P) = 0$, i.e. when $P \in \supp(\Theta)$ (see \cite[Lemma 3.2.4]{tripThesis}). Furthermore, on $J_{\boldsymbol{c}} \setminus \supp(\Theta)$, we have by construction that \[\wp^{(\boldsymbol{c})}(\boldsymbol{t}(P)) \neq 0.\] Hence \[w_1(P)\wp^{(\boldsymbol{c})}(\boldsymbol{t}(P)) \neq 0.\]

\begin{theorem}\label{thm:generalnaiveheight}
    Let $P \in J_{\boldsymbol{c}} \setminus \supp(\Theta)$. Then \[\lim_{n \ra \infty}\frac{1}{p^{2n}} H_p^{(\boldsymbol{c})}(p^n P) = h_p^{(\boldsymbol{c})}(P).\]
    In particular, the limit exists and can be extended to a quadratic height function on $J(\Q)$.
\end{theorem}

\begin{proof}
    A similar argument as in Proposition \ref{prop:standardformhhat} shows that for $Q \in J_{\boldsymbol{c}} \setminus \supp(\Theta)$, we have \[h_p^{(\boldsymbol{c})}(Q) = - \log_p\left(\frac{\sigma_p^{(\boldsymbol{c})}(\bt(Q))^2}{w_1(Q)}\right).\] Rewriting this using the definition of $H_p^{(\boldsymbol{c})}$ and $\wp^{(\boldsymbol{c})}$ gives
    \[h_p^{(\boldsymbol{c})}(Q) = - \log_p\left(\sigma_p^{(\boldsymbol{c})}(\bt(Q))^2 \wp^{(\boldsymbol{c})}(Q)\right) + H_p^{(\boldsymbol{c})}(Q).\] 
    It follows from Lemma \ref{lemma:mostgenerallemma}, using our observations on the convergence of $\sigma_p^{(\boldsymbol{c})}(\boldsymbol{t})^2 \wp^{(\boldsymbol{c})}(\boldsymbol{t})$ together with Corollary \ref{cor:ordtipnP} and Lemma \ref{lem:S(P)infinite}, that \[\lim_{n \ra \infty}\frac{1}{p^{2n}} H_p^{(\boldsymbol{c})}(p^n P) = h_p^{(\boldsymbol{c})}(P)\] for $P \in J_{\boldsymbol{c}} \setminus \supp(\Theta)$.
\end{proof}

This gives an alternative construction for the quadratic heights in \cite{bianchi3} for any choice of $\boldsymbol{c} = (c_{ij})_{i, j} \in M_{2 \times 2}(\Q_p)$ with $c_{12} = c_{21}$, including the canonical choice of Blakestad.

{\small
\printbibliography[heading=bibintoc]}

@book{silverman1,
  title={The arithmetic of elliptic curves},
  author={Silverman, J. H.},
  SERIES = {Graduate Texts in Mathematics},
  EDITION = {Second Edition},
  volume={106},
  year={2009},
  publisher={Springer}
}

@book{silverman2,
  title={Advanced topics in the arithmetic of elliptic curves},
  author={Silverman, J. H.},
   SERIES = {Graduate Texts in Mathematics},
  volume={151},
  year={1994},
  PUBLISHER = {Springer-Verlag},
}

@article{uchida,
  title={Canonical local heights and multiplication formulas for the {J}acobians of curves of genus 2},
  author={Uchida, Y.},
  JOURNAL = {Acta Arith.},
  volume={149},
  number={2},
  pages={111--130},
  year={2011}
}

@article{flynnsmart,
  title={Canonical heights on the {J}acobians of curves of genus {$2$}
              and the infinite descent},
  author={Flynn, E. V. and Smart, N. P.},
  journal = {Acta Arith.},
  fjournal={Acta Arithmetica},
     VOLUME = {79},
      YEAR = {1997},
    NUMBER = {4},
     PAGES = {333--352},
}

@article{grant,
  title={Formal groups in genus two},
  author={Grant, D.},
  year={1990},
  JOURNAL = {J. Reine Angew. Math.},
  FJOURNAL = {Journal f\"{u}r die Reine und Angewandte Mathematik},
  VOLUME = {411},
     PAGES = {96--121},
     YEAR = {1990}
}

@article{flynn1,
  title={The {J}acobian and formal group of a curve of genus {$2$} over
              an arbitrary ground field},
  author={Flynn, E. V.},
  journal = {Math. Proc. Cambridge Philos. Soc.},
  fjournal={Mathematical Proceedings of the Cambridge Philosophical
              Society},
  volume={107},
  number={3},
  pages={425--441},
  year={1990}
}

@article{flynn2,
  title={The group law on the {J}acobian of a curve of genus {$2$}},
  author={Flynn, E. V.},
  journal = {J. Reine Angew. Math.},
  fjournal={Journal f\"{u}r die Reine und Angewandte Mathematik},
  year={1993},
  VOLUME = {439},
  PAGES = {45--69}
}

@misc{biquadratic, 
    title={Biquadratic Forms},
    howpublished = {\url{https://people.maths.ox.ac.uk/flynn/genus2/kummer/biquadratic.forms}},
    author={Flynn, E. V.},
    note={Accessed: 12 November 2021}
}

@incollection{perrinriou,
  title={Hauteurs $p$-adiques [$p$-adic heights]},
  author={Perrin-Riou, B.},
  booktitle={S\'{e}minaire de th\'{e}orie des nombres, Paris 1982–-83},
  SERIES = {Progr. Math.},
  VOLUME = {51},
  year={1984},
  pages={233--257},
  PUBLISHER = {Birkh\"{a}user}
}

@book{casselsflynn,
  title={Prolegomena to a middlebrow arithmetic of curves of genus {$2$}},
  author={Cassels, J. W. S. and Flynn, E. V.},
  SERIES = {London Mathematical Society Lecture Note Series},
  volume={230},
  year={1996},
  publisher={Cambridge University Press}
}

@phdthesis{blakestad,
  title={On Generalizations of {$p$}-Adic Weierstrass Sigma and Zeta
    Functions},
  author={Blakestad, C.},
  year={2018},
  school={University of Colorado}
}

@incollection{bernardi,
  TITLE = {Hauteur {$p$}-adique sur les courbes elliptiques [{$p$-adic height on elliptic curves}]},
  BOOKTITLE = {S\'{e}minaire de {T}h\'{e}orie des {N}ombres, {P}aris 1979--80},
  author={Bernardi, D.},
  SERIES = {Progr. Math.},
  year={1981},
  volume={12},
  pages={1--14},
  PUBLISHER = {Birkh\"{a}user}
}

@article{perrinriou2,
  title={Sur les hauteurs $p$-adiques [On $p$-adic heights]},
  author={Perrin-Riou, B.},
  journal={C. R. Math. Acad. Sci. Paris},
  volume={296},
  number = {6},
  year={1983},
  pages={291--294}
}

@article{stollheight,
  author={Stoll, M.},
  journal={Acta Arith.},
  TITLE = {On the height constant for curves of genus two. {II}},
  VOLUME = {104},
  YEAR = {2002},
  NUMBER = {2},
  PAGES = {165--182}
}

@article{neron,
  title={Quasi-fonctions et hauteurs sur les vari{\'e}t{\'e}s ab{\'e}liennes},
  author={N{\'e}ron, A.},
  Journal = {Ann. Math.},
  FJOURNAL = {Annals of Mathematics. Second Series},
  volume={82},
  number = {2},
  pages={249--331},
  year={1965}
}

@book{lang,
  title={Fundamentals of Diophantine geometry},
  author={Lang, S.},
  year={1983},
  publisher={Springer-Verlag}
}

@incollection{neron2,
    author = {N\'{e}ron, A.},
    title = {Fonctions th\^{e}ta {$p$}-adiques et hauteurs {$p$}-adiques [$p$-adic theta functions and $p$-adic heights]},
    BOOKTITLE = {{S}éminaire de {T}héorie des {N}ombres, {P}aris 1980-81},
    series = {Progr. Math.},
    volume = {22},
    pages = {149--174},
    publisher = {Birkh\"{a}user},
      year = {1982}
}

@article {mazurtate,
    author = {Mazur, B. and Tate, J.},
    title = {The {$p$}-adic sigma function},
   journal = {Duke Math. J.},
   fjournal = {Duke Mathematical Journal},
    volume = {62},
    year = {1991},
    number = {3},
    pages = {663--688}
}

@article {balakrishnan,
    author = {Balakrishnan, J. S. and M\"{u}ller, J. S. and Stein,
              W. A.},
    title = {A {$p$}-adic analogue of the conjecture of {B}irch and
              {S}winnerton-{D}yer for modular abelian varieties},
   journal = {Math. Comp.},
    fjournal = {Mathematics of Computation},
    volume = {85},
    year = {2016},
    number = {298},
    pages = {983--1016}
}

@book {birkenhake,
    AUTHOR = {Birkenhake, C. and Lange, H.},
     TITLE = {Complex abelian varieties},
    SERIES = {Grundlehren der mathematischen Wissenschaften},
    VOLUME = {302},
   EDITION = {Second edition},
 PUBLISHER = {Springer-Verlag},
      YEAR = {2004}
}

@article {balakrishnanChabauty,
    AUTHOR = {Balakrishnan, J. S. and Dogra, N.},
     TITLE = {Quadratic {C}habauty and rational points, {I}: {$p$}-adic
              heights},
    JOURNAL = {Duke Math. J.},
    FJOURNAL = {Duke Mathematical Journal},
    VOLUME = {167},
      YEAR = {2018},
    NUMBER = {11},
     PAGES = {1981--2038}
}

@article{teitelbaum,
    AUTHOR = {Mazur, B. and Tate, J. and Teitelbaum, J.},
     TITLE = {On {$p$}-adic analogues of the conjectures of {B}irch and
              {S}winnerton-{D}yer},
    JOURNAL = {Invent. Math.},
    FJOURNAL = {Inventiones Mathematicae},
    VOLUME = {84},
      YEAR = {1986},
    NUMBER = {1},
     PAGES = {1--48}
}

@misc{bianchi3,
    author = {Bianchi, F.},
    title = {{$p$}-adic sigma functions and heights on {J}acobians of genus {$2$} curves},
    eprint = {2302.03454v1},
    archivePrefix = {arXiv},
    primaryClass = {math.NT},
    year = {2023}
}

@mastersthesis{tripThesis,
  title={A naive $p$-adic height function on the Jacobians of curves of genus 2},
  author={Trip, M. T.},
  year={2022},
  school={University of Groningen},
  type={Master Thesis}
}

@article {BBBM,
    AUTHOR = {Balakrishnan, J. S. and Besser, A. and Bianchi,
              F. and M\"{u}ller, J. S.},
     TITLE = {Explicit quadratic {C}habauty over number fields},
   JOURNAL = {Israel J. Math.},
  FJOURNAL = {Israel Journal of Mathematics},
    VOLUME = {243},
      YEAR = {2021},
    NUMBER = {1},
     PAGES = {185--232}
}

@article {Faltings1,
    AUTHOR = {Faltings, G.},
     TITLE = {Endlichkeitss\"{a}tze f\"{u}r abelsche {V}ariet\"{a}ten
              \"{u}ber {Z}ahlk\"{o}rpern [Finiteness theorems for abelian varieties over number fields]},
   JOURNAL = {Invent. Math.},
  FJOURNAL = {Inventiones Mathematicae},
    VOLUME = {73},
      YEAR = {1983},
    NUMBER = {3},
     PAGES = {349--366}
}

@article {Faltings2,
    AUTHOR = {Faltings, G.},
     TITLE = {Erratum: ``Endlichkeitss\"{a}tze f\"{u}r abelsche {V}ariet\"{a}ten
              \"{u}ber {Z}ahlk\"{o}rpern''},
   JOURNAL = {Invent. Math.},
  FJOURNAL = {Inventiones Mathematicae},
    VOLUME = {75},
      YEAR = {1984},
    NUMBER = {2},
     PAGES = {381}
}

@book{baker,
  title={An introduction to the theory of multiply periodic functions},
  author={Baker, H. F.},
  year={1907},
  publisher={Cambridge University Press}
}

@book{hazewinkel,
  title={Formal groups and applications},
  author={Hazewinkel, M.},
  year={1978},
  publisher={Academic Press}
}

@article {mattuck,
    AUTHOR = {Mattuck, A.},
     TITLE = {Abelian varieties over {$p$}-adic ground fields},
   JOURNAL = {Ann. of Math.},
  FJOURNAL = {Annals of Mathematics},
    VOLUME = {62},
    NUMBER = {1},
      YEAR = {1955},
     PAGES = {92--119}
}

@article {cantor,
    AUTHOR = {Cantor, D. G.},
     TITLE = {On the analogue of the division polynomials for hyperelliptic
              curves},
   JOURNAL = {J. Reine Angew. Math.},
  FJOURNAL = {Journal f\"ur die Reine und Angewandte Mathematik. [Crelle's
              Journal]},
    VOLUME = {447},
      YEAR = {1994},
     PAGES = {91--145}
}

\end{document}